\documentclass[lettersize,journal]{IEEEtran}
\usepackage{algorithm}
\usepackage{array}
\usepackage[caption=false,font=normalsize,labelfont=sf,textfont=sf]{subfig}
\usepackage{textcomp}
\usepackage{stfloats}
\usepackage{url}
\usepackage{verbatim}
\usepackage{graphicx}
\usepackage{cite}
\usepackage{amsmath,amssymb,amsfonts,amsthm,bm,mathtools}
\usepackage{algorithmic}
\usepackage{xcolor}
\usepackage{multirow}
\usepackage{booktabs}
\graphicspath{{fig/}}
\usepackage{siunitx}
\usepackage[hidelinks]{hyperref}

\newtheorem{theo}{Theorem}
\newtheorem{lemma}[theo]{Lemma}
\newtheorem{prop}[theo]{Proposition}
\newtheorem{coro}[theo]{Corollary}
\newtheorem{defi}[theo]{Definition}

\newtheorem{propprime}[theo]{Proposition}
\theoremstyle{remark}
\newtheorem{rem}[theo]{Remark}

\newif\ifincludeproofs

\newcommand{\E}{\mathbb{E}}
\newcommand{\Var}{\operatorname{Var}}

\newcommand{\MixPoi}{\operatorname{MixPoi}}

\newcommand{\deff}{d_{\mathrm{eff}}}

\begin{document}

\title{Spectral Costs of Multinomial and Residual Resampling in High-Dimensional Particle Ensembles: A Free-Probability Analysis}

\author{SONG~Deming,~\IEEEmembership{Student Member,~IEEE,}%
\thanks{SONG Deming is with the Faculty of Applied Sciences, Macao Polytechnic University, Macao, China (e-mail: p2514650@mpu.edu.mo).}%
\thanks{Corresponding author: Chi-Kin Lam (e-mail: cklamsta@mpu.edu.mo).}%
\hspace{0.5 cm} 
\and
\hspace{0.5 cm}
Chi-Kin Lam%
\thanks{Chi-Kin Lam is with the Faculty of Applied Sciences, Macao Polytechnic University, Macao, China.}}%


\maketitle

\begin{abstract}
Resampling is the standard remedy against weight degeneracy in particle filters, yet its effect on the spectrum of the particle second-moment matrix is poorly understood. We study a single resampling step for a $d$-dimensional ensemble of $N$ particles with $d/N\to c$. For any unbiased scheme, the resampled empirical second-moment matrix dominates its pre-resampling counterpart in conditional convex order, while the conditional mean is preserved; the excess second spectral moment equals the Frobenius energy of the resampling perturbation. For multinomial resampling, the empirical count law converges to the mixed-Poisson law determined by the limiting weight law. For residual resampling, the count law is determined by the joint limit of floor--fraction pairs: we classify all residual limits of weight sequences converging to $\delta_1$ in law as a one-parameter family between the deterministic identity and the full multinomial count law. Each count law feeds a common covariance-side spectral map, yielding exact consequences: a closed-form zero-eigenvalue mass, a rank-and-spectral-gap transition with a universal lower-edge benchmark, computable regular edges, unbounded right support, and scheme-specific second-moment costs. The effective sample size determines one bulk statistic but cannot determine the limiting spectrum. Pre-registered numerical experiments confirm the theory at every level.
\end{abstract}


\section{Introduction}
\label{sec:intro}

\IEEEPARstart{P}{article} filters approximate posterior distributions by weighted particles~\cite{doucet2011tutorial,delmoral2004feynman}. Resampling, the redrawing of particles according to their weights, is the
standard remedy against weight degeneracy\cite{kong1994sequential}. Its effect on the geometry of the
ensemble is considerably less well understood: specifically, its effect on the
spectrum of the ensemble second-moment matrix, which governs the conditioning
of subsequent computations, and on the rank of the ensemble subspace. This
paper isolates a single resampling step and determines how resampling
transforms the limiting spectral distribution of the empirical second-moment
matrix in a high-dimensional particle ensemble.

\subsection{Finite-sample results}

For any unbiased scheme, the resampled matrix dominates the pre-resampling
matrix in conditional convex order (Theorem~\ref{thm:jensen}). The principal
finite-sample result is Proposition~\ref{prop:cost}, which identifies the
excess second spectral moment with the expected Frobenius energy of the
resampling perturbation. The resulting expression is a quadratic form in the
count covariance, and it separates the contribution of the particle geometry,
carried by $G\circ G$, from that of the mechanism noise, carried by
$\operatorname{Cov}(\bm{n}\mid X_N,w)$.

\subsection{High-dimensional limit}

The covariance-side step is classical: weighted sample-covariance matrices
converge to a deterministic law $\mathcal{T}_c(\nu)$ characterized by a
Silverstein-type fixed-point equation, equivalently by a
free-multiplicative-convolution relation (Definition~\ref{def:Tc}). The
contributions of this paper concern the count-law side.

For multinomial resampling the limiting count law is determined by the
limiting weight law $\rho$ alone (Theorem~\ref{thm:multi-count}). For residual
resampling it is not. The floor--fraction map is discontinuous at the
integers, and the limiting count law is determined instead by the joint limit
$\pi$ of floor--fraction pairs (Theorem~\ref{thm:res-count}). The strongest
form of this dependence is a complete classification
(Theorem~\ref{thm:classification}): the residual limits of weight sequences
converging to $\delta_1$ constitute the one-parameter family
$(1-\alpha)\delta_1+\alpha\operatorname{Poi}(1)$ with $\alpha\in[0,1]$,
ranging from the deterministic identity at $\alpha=0$ to the full multinomial
count law at $\alpha=1$. The endpoint $\alpha=1$ is attained by a bounded
two-level construction, which establishes that the phenomenon originates at
the integer boundary rather than in heavy tails.

The resulting framework is
\begin{equation}
  \mu_{\mathsf{s},H} = H\boxtimes\mathcal{T}_c(\nu_{\mathsf{s}}),
  \label{eq:framework}
\end{equation}
with
\begin{equation}
  \nu_{\mathrm{pre}} = \rho, \quad
  \nu_{\mathrm{multi}} = \operatorname{MixPoi}(\rho), \quad
  \nu_{\mathrm{res}} = \mathcal{L}_\pi(K+\operatorname{Poi}(U)),
  \label{eq:definitions}
\end{equation}
where the last of these depends on finer information than $\rho$ alone.

\subsection{Consequences}

Composed with the covariance-side spectral map, these count laws yield a
closed-form zero-eigenvalue mass; a rank-and-spectral-gap transition at
$p_{\mathsf{s}}=1-\nu_{\mathsf{s}}(\{0\})$, with an explicit
Marchenko--Pastur lower-edge benchmark\cite{marchenko1967distribution} and computable regular edges; exact
second-moment increments $c\,a_1^2$ for multinomial resampling and
$c\,a_1^2\,\mathbb{E}_\pi U$ for residual resampling; an asymptotic
rank-preservation ordering $\mu_{\mathrm{res}}(\{0\})\le
\mu_{\mathrm{multi}}(\{0\})$; and a finite-sample Frobenius ordering of the
two schemes. For multinomial resampling the limiting right support is
unbounded, so that the pseudo-condition number diverges in probability and
bulk extent must be interpreted as a quantile notion; residual limits may
instead have bounded support. Sections~\ref{sec:lsd}--\ref{sec:res} state
these results precisely.

The results order the spectral observables by the information each requires.
The bulk second moment, and therefore the effective sample size (ESS), depends
only on $r_2:=\int\lambda^2\rho$. The zero mass and the rank threshold depend
only on $\mathcal{L}_\rho(1)$. The spectral edges and the full limiting
spectral distribution depend on the entire count law. The choice of mechanism
is therefore not observable through ESS\cite{martino2017ess}.

\subsection{Relation to prior work}

Convex-order comparisons of resampling schemes are due to Gerber, Chopin and
Whiteley \cite{gerber2019negative}, and the inadequacy of ESS as a sole
diagnostic was argued by Elvira, Martino and Robert
\cite{elvira2022rethinking}. Weighted and separable covariance models are
classical \cite{silverstein1995strong,bai2010spectral,hachem2007deterministic,%
zhang2006spectral,paul2009no}, and deterministic resampling schemes were
surveyed by Li, Sattar and Sun \cite{li2012deterministic}. The bootstrap
spectral analysis of El Karoui and Purdom \cite{elkaroui2016bootstrap}
concerns the statistical bootstrap rather than the bootstrap particle filter\cite{gordon1993novel, cappe2007overview}.
The elements introduced here are the count-law limits of resampling
mechanisms, including the integer-boundary classification for residual
resampling, together with the mechanism-level spectral consequences derived
from them.

\subsection{Notation}

Eigenvalues are ordered $\lambda_1(M)\ge\cdots\ge\lambda_d(M)$, and
$\lambda_{\min}^+$ denotes the smallest positive eigenvalue;
$\mathcal{S}_+^d$ denotes the positive semidefinite matrices and
$\operatorname{tr}$ the trace. Further, $\delta_x$ is the Dirac mass,
$\Rightarrow$ weak convergence, $\mathrm{MP}(y)$ the Marchenko--Pastur law\cite{marchenko1967distribution}
with ratio $y$ and mean $1$, $\boxtimes$ free multiplicative convolution,
$D_c$ the push-forward under $x\mapsto cx$, $G=X^\top X$ the Gram matrix,
$A\circ B$ the entrywise product, $\nu^{(M)}:=(t\mapsto t\wedge M)_\#\nu$, and
$\kappa^+(S):=\lambda_{\max}(S)/\lambda_{\min}^+(S)$. We write
$\operatorname{frac}(x):=x-\lfloor x\rfloor\in[0,1)$, $\mathcal{L}_\rho$ for
the Laplace transform of $\rho$, and $\mathcal{L}(\cdot)$ for the law of a
random variable. Finally, $c_N:=d/N$, $r_2:=\int\lambda^2\rho(d\lambda)$, and
$\pi$ denotes a probability law on $\mathbb{Z}_{\ge0}\times[0,1]$ whose second
coordinate is permitted to attain $1$, with
$\mathbb{E}_\pi U:=\int u\,\pi(dk,du)$.

\section{One-Step Resampling Model, Assumptions, and Conventions}
\label{sec:model}

\subsection{Particles}

Let $X_N=[x_1,\dots,x_N]\in\mathbb{R}^{d\times N}$ with
$X_N=\Sigma_N^{1/2}Z_N$, where $\Sigma_N\in\mathcal{S}_+^d$ is a
deterministic population second-moment matrix with empirical spectral law
$H_N$, and where $Z_N$ has independent and identically distributed entries
drawn from a fixed distribution with mean $0$, variance $1$, and the moment
conditions stated with each result.

The matrix $\Sigma_N$ represents the deterministic correlation structure,
which in filtering terminology is the background or forecast covariance; its
spectral law $H_N$ is the free parameter of the asymptotic analysis. The
matrix $Z_N$ is a whitened fluctuation matrix, so that the columns
$x_i=\Sigma_N^{1/2}z_i$ are independent, identically distributed and centered,
with $\mathbb{E}[x_ix_i^\top]=\Sigma_N$. The separable covariance model
assigns the population spectrum $H$ and the count law $\nu$ to two separate
channels, which is what permits the free-probability analysis. No particle
model is required for the finite-sample results of
Section~\ref{sec:finite}; the model is invoked only where spectra are
concerned.

\subsection{Conditioning conventions}

The finite-sample statements of Section~\ref{sec:finite} are conditional on
$(X_N,w)$. Asymptotic statements are taken under the joint randomness of the
particle matrix and the resampling mechanism.

\subsection{Weights and exogeneity}

Let $w=(w_1,\dots,w_N)$ with $w_i\ge0$ and $\sum_iw_i=1$, and set
$\lambda_i:=Nw_i$ and $\rho_N:=\frac1N\sum_i\delta_{\lambda_i}$. Throughout,
exogeneity is understood as follows. For the asymptotic analysis the weights
are either deterministic, which is the convention adopted in the main text, or
random and independent of $Z_N$, in which case all weight-law convergences
hold in probability and the limit laws are deterministic. Multinomial and
residual resampling employ auxiliary randomness independent of $Z_N$, so that
the count vector is independent of the particle matrix conditionally on the
weights. No such independence is required for the finite-sample results.

\subsection{Counts}

A resampling scheme produces $\bm{n}=(n_1,\dots,n_N)\in\mathbb{Z}_{\ge0}^N$
satisfying
\begin{itemize}
  \item[(A1)] $\sum_in_i=N$ almost surely, and
  \item[(A2)] $\mathbb{E}[n_i\mid X_N,w]=Nw_i$.
\end{itemize}

\subsection{Second-moment matrices}

The weighted (pre-resampling) and resampled (post-resampling) empirical
second-moment matrices are
\begin{align}
  S_{w,N}&=\sum_iw_ix_ix_i^\top=\tfrac1N X_N\Lambda_NX_N^\top,\nonumber\\
  S_{r,N}&=\tfrac1N\sum_in_ix_ix_i^\top=\tfrac1N X_ND_NX_N^\top,\nonumber
\end{align}
with $\Lambda_N=\operatorname{diag}(\lambda_i)$ and
$D_N=\operatorname{diag}(n_i)$. Both matrices are uncentered. Centering
perturbs each matrix by a term of rank at most one and therefore leaves the
limiting spectral distribution unchanged as $d\to\infty$ whenever that
distribution exists. The finite-sample conditional inequalities are, however,
sensitive to centering, since
$\mathbb{E}[S_{r,\mathrm{centered}}\mid X_N,w]
=S_{w,\mathrm{centered}}-\operatorname{Cov}(\bar{x}_r\mid X_N,w)$.

\subsection{Asymptotic assumptions}

\begin{itemize}
  \item[(A3)] $d/N\to c\in(0,\infty)$.
  \item[(A4)] For multinomial resampling, $\rho_N\Rightarrow\rho$ with
    $\int\lambda^2\rho_N\to\int\lambda^2\rho<\infty$.
  \item[(A4R)] For residual resampling,
    $\pi_N:=\frac1N\sum_i\delta_{(\lfloor\lambda_i\rfloor,
    \operatorname{frac}(\lambda_i))}\Rightarrow\pi$ on
    $\mathbb{Z}_{\ge0}\times[0,1]$, with
    $\int\lambda^2\rho_N\to\int\lambda^2\rho<\infty$, where $\rho$ is the
    push-forward of $\pi$ under $(k,u)\mapsto k+u$.
  \item[(A5)] $H_N\Rightarrow H$ with $H$ compactly supported and
    $\sup_N\|\Sigma_N\|_{\mathrm{op}}<\infty$; we set
    $a_1:=\int\tau H(d\tau)$ and $m_{2,H}:=\int\tau^2H(d\tau)$.
\end{itemize}

We record two consequences of these assumptions. The second-moment
assumption implies $\max_iw_i\le(\sum_i\lambda_i^2)^{1/2}/N=O(N^{-1/2})\to0$,
so that diffuseness of the weights is automatic and need not be imposed
separately. Finiteness of the limiting first moments follows by uniform
integrability, giving $\int\lambda\,\rho(d\lambda)=1$ and
$\mathbb{E}_\pi(K+U)=1$. Assumption (A4) does not imply (A4R), because the
floor--fraction map is discontinuous at the positive integers; see
Theorem~\ref{thm:classification} and Section~\ref{sec:res}.

\subsection{Moment requirements}

The conditional identities of Section~\ref{sec:finite} require no particle
model. Convergence of the limiting spectral distribution requires independent,
identically distributed entries with finite fourth moment. Convergence of the
empirical second spectral moments requires, in addition, the count-tail
control established in the supplementary material, or else Gaussianity.
Extreme-eigenvalue statements invoke the extreme-spectrum theorem with entries
from a fixed distribution.

\section{Finite-Sample Spectral Inequalities}
\label{sec:finite}

All statements in this section are conditional on $(X_N,w)$, hold for every
$d$ and $N$, and apply to every scheme satisfying (A1)--(A2).

\begin{lemma}[Conditional unbiasedness identity]
\label{lem:unbiased}
$\mathbb{E}[S_{r,N}\mid X_N,w]=S_{w,N}$.
\end{lemma}

\begin{defi}[Conditional convex order\cite{voiculescu1987multiplication,bercovici1993free,rockafellar1970convex}]
\label{def:cxorder}
We write $S_{w,N}\preceq_{\rm cx}S_{r,N}$ conditionally on $(X_N,w)$ if
$F(S_{w,N})\le\mathbb{E}[F(S_{r,N})\mid X_N,w]$ almost surely for every convex
$F:\mathcal{S}_+^d\to\mathbb{R}$ for which the required expectations exist.
Convexity is understood in the ordinary affine sense on the positive
semidefinite cone; no operator convexity or L\"owner-order comparison is
asserted.
\end{defi}

\begin{theo}[Scheme-agnostic Jensen principle\cite{hansen1982jensen}]
\label{thm:jensen}
Under (A1)--(A2), $S_{w,N}\preceq_{\rm cx}S_{r,N}$ conditionally on
$(X_N,w)$.
\end{theo}

Theorem~\ref{thm:jensen} applies in particular to
$F(M)=\operatorname{tr}\varphi(M)$ for scalar convex $\varphi$, by Davis trace
convexity\cite{davis1957schwarz}. Admissible choices include $\varphi=x^q$ for all $q\ge1$ and
$\varphi=x\log x$, together with $F=\lambda_{\max}$ and
$F(M)=-\log\det(M+\varepsilon I)$ for $\varepsilon>0$; the last of these
states that the log-determinant decreases in conditional expectation. We note
that $x^q$ with $q>2$ is not operator convex on $(0,\infty)$, whereas trace
convexity remains available.

\begin{prop}[Exact second spectral moment cost\cite{bhatia1997matrix,marshall1979inequalities}]
\label{prop:cost}
With $G=X_N^\top X_N$, for any scheme satisfying (A1)--(A2),
\begin{align}
  &\mathbb{E}\big[\operatorname{tr}S_{r,N}^2-\operatorname{tr}S_{w,N}^2
    \mid X_N,w\big] \nonumber\\
  &\quad=\mathbb{E}\big[\|S_{r,N}-S_{w,N}\|_F^2 \mid X_N,w\big] \nonumber\\
  &\quad=\tfrac{1}{N^2}\operatorname{tr}\big[(G\circ G)
    \operatorname{Cov}(\bm{n}\mid X_N,w)\big]\ \ge\ 0.
  \label{eq:cost-identity}
\end{align}
The second expression admits an error-energy interpretation and renders the
nonnegativity immediate. For multinomial resampling,
\begin{align}
  &\mathbb{E}\big[\operatorname{tr}S_{r,N}^2-\operatorname{tr}S_{w,N}^2
    \mid X_N,w\big] \nonumber\\
  &\quad=\tfrac{1}{N}\sum_i w_i \|x_i x_i^\top - S_{w,N}\|_F^2 \nonumber\\
  &\quad=\tfrac{1}{N}\Big[\sum_i w_i \|x_i\|^4
    - \operatorname{tr}S_{w,N}^2\Big] \nonumber\\
  &\quad\ge \tfrac{1}{N}\operatorname{Var}_w\big(\|x_i\|^2\big)\ \ge\ 0.
  \label{eq:cost-multi}
\end{align}
\end{prop}

\begin{rem}[Vanishing of the cost]
\label{rem:vanish}
The cost is zero precisely when resampling leaves the second-moment matrix
unchanged almost surely. Two distinct mechanisms produce this situation:
deterministic counts, as arise for residual resampling when every
$\lambda_i\in\mathbb{Z}_{\ge0}$, or for stratified, systematic and residual
resampling under uniform weights; and degenerate particle configurations. For
multinomial resampling, the geometric form \eqref{eq:cost-multi} vanishes if
and only if $x_ix_i^\top=S_{w,N}$ for every $i$ with $w_i>0$, a condition that
covers the all-zero and the $\pm v$ collinear configurations.
\end{rem}

\begin{propprime}[Population-averaged cost and diffuse-weight limit]
\label{prop:popcost}
Under multinomial resampling, independent and identically distributed entries
with kurtosis $\kappa_4:=\mathbb{E}Z_{11}^4-3$, and weights independent of the
particle matrix, we have exactly, with $s_{2,N}:=\sum_iw_i^2$,
\begin{equation}
  \begin{aligned}
    &\tfrac{1}{d}\mathbb{E}\big[\operatorname{tr}S_{r,N}^2
      -\operatorname{tr}S_{w,N}^2 \mid w\big]\\
    &= \tfrac{1-s_{2,N}}{Nd}\Big[(\operatorname{tr}\Sigma_N)^2
      + \operatorname{tr}\Sigma_N^2
      + \kappa_4\textstyle\sum_{j=1}^d (\Sigma_N)_{jj}^2\Big].
  \end{aligned}
  \label{eq:popcost}
\end{equation}
Under (A3) and (A5) the bracketed terms satisfy
$\operatorname{tr}\Sigma_N^2=O(d)$ and $\sum_j(\Sigma_N)_{jj}^2=O(d)$, whence
\begin{align}
  \tfrac1d\mathbb{E}[\mathrm{cost}\mid w]
    &=c_N\big(\tfrac{\operatorname{tr}\Sigma_N}{d}\big)^2(1-s_{2,N})
      +O(N^{-1})\nonumber\\
    &=c\,a_1^2(1-s_{2,N})+o(1),
  \label{eq:popcost-asym}
\end{align}
which reduces to $c\,a_1^2$ in the diffuse regime. The statement is specific
to multinomial resampling; residual resampling under exactly uniform weights,
for instance, leaves the matrix unchanged.
\end{propprime}

\section{Count Laws and the Spectral Map}
\label{sec:count}

For a probability law $\rho$ on $[0,\infty)$ with mean $1$, the mixed-Poisson
law is defined by
\begin{equation}
  \operatorname{MixPoi}(\rho):\quad
  \nu(\{k\})=\int_0^\infty e^{-\lambda}\tfrac{\lambda^k}{k!}\,\rho(d\lambda),
  \quad k\in\mathbb{Z}_{\ge0},
  \label{eq:mixpoi}
\end{equation}
and has mean $1$, zero mass $\mathcal{L}_\rho(1)$, and second moment
$b_2:=\int k^2\nu(dk)=1+r_2$.

\begin{theo}[Multinomial count law]
\label{thm:multi-count}
Under (A4), for multinomial resampling, in probability,
\begin{equation}
  \nu_N:=\tfrac1N\sum_i\delta_{n_i}\Rightarrow\nu=\operatorname{MixPoi}(\rho),
  \qquad \tfrac1N\sum_in_i^2\to b_2.
  \label{eq:multi-count}
\end{equation}
Exactly uniform weights, and more generally any weight sequence satisfying
(A4) with $\rho=\delta_1$, yield $\nu=\operatorname{Poi}(1)$, so that a
fraction $e^{-1}$ of the particles is not selected.
\end{theo}

The transfer of the classical theorem to a random and unbounded count diagonal
proceeds by a Poissonization coupling together with truncation and rank
control; the coupling controls weak spectral convergence but does not by itself
control second spectral moments. Both steps are stated and proved in the
supplementary material.

\begin{defi}[Covariance-side spectral map\cite{silverstein1995analysis}]
\label{def:Tc}
For a probability law $\eta$ on $[0,\infty)$, define $\mathcal{T}_c(\eta)$ as
the probability law determined by
\begin{equation}
  c\,\mathcal{T}_c(\eta)+(1-c)\,\delta_0
    = D_c\big(\eta\boxtimes\mathrm{MP}(1/c)\big),
  \label{eq:Tc}
\end{equation}
or equivalently by the fixed-point equation
\begin{align}
  z&=-\tfrac1{m(z)}+\int\tfrac{t}{1+c\,t\,m(z)}\,\eta(dt),\nonumber\\
  m(z)&=\int\tfrac{\mathcal{T}_c(\eta)(dx)}{x-z},\qquad z\in\mathbb{C}^+,
  \label{eq:fp}
\end{align}
taken with the branch of a probability-measure Stieltjes transform, namely
$m:\mathbb{C}^+\to\mathbb{C}^+$ with $\lim_{v\to\infty}iv\,m(iv)=-1$. Mean one
is not required of $\eta$; the truncated count laws $\nu^{(M)}$ have mean at
most one and are never rescaled.
\end{defi}

Three properties of Definition~\ref{def:Tc} are used repeatedly. First,
\eqref{eq:Tc} is an identity of measures rather than a convex-mixture
representation when $c>1$; equivalently
$\mathcal{T}_c(\eta)=\frac1cD_c(\eta\boxtimes\mathrm{MP}(1/c))
+(1-\frac1c)\delta_0$ as measures, the negative coefficient for $c<1$ being
absorbed by the zero atom of the product law, in accordance with
Proposition~\ref{prop:atom}. Second, $\mathcal{T}_c$ is injective on all
probability laws on $[0,\infty)$. Third, for numerical evaluation the
companion-side transform is best computed from $m$ through the auxiliary
equation $z=-c/u+\int t(1+tu)^{-1}\eta(dt)$ with $u=cm$, followed by
$e=u+(c-1)/z$. The classical companion fixed-point equation governs a
different object, namely the companion transform of the $W$-side model, and
several mutually incompatible forms coincide at $\eta=\delta_1$. Recovery of
the Marchenko--Pastur law alone therefore does not validate a solver.

\begin{theo}[Residual count law]
\label{thm:res-count}
Under (A4R), for residual resampling with
$n_i=\lfloor\lambda_i\rfloor+\tilde n_i$, where
$\tilde{\bm{n}}\sim\mathrm{Multinomial}(R_N;u/R_N)$ with
$u_i:=\operatorname{frac}(\lambda_i)$ and $R_N:=\sum_iu_i$ when $R_N>0$, and
$\tilde{\bm{n}}:=0$ when $R_N=0$, we have in probability
\begin{equation}
  \nu_N\Rightarrow\nu_{\mathrm{res}}:=\mathcal{L}(K+P_U),
  \qquad (K,U)\sim\pi,
  \label{eq:res-count}
\end{equation}
where $P_U\mid(K,U)\sim\operatorname{Poi}(U)$ is a conditional Poisson
variable rather than an independent one of fixed parameter, and
$\frac1N\sum_in_i^2\to b_2(\nu_{\mathrm{res}})=r_2+\mathbb{E}_\pi U$. All
regimes of $R_N$ are covered, namely $R_N=0$, $R_N\to\infty$ with $R_N=o(N)$,
and $R_N\asymp N$.
\end{theo}

\begin{theo}[Classification of residual limits under $\rho=\delta_1$]
\label{thm:classification}
If $\rho_N\Rightarrow\delta_1$ and (A4R) holds with limit $\pi$, then $\pi$ is
supported on $\{K+U=1\}$, which within $\mathbb{Z}_{\ge0}\times[0,1]$ consists
of the two points $(1,0)$ and $(0,1)$. Consequently
\begin{align}
  \pi&=(1-\alpha)\delta_{(1,0)}+\alpha\delta_{(0,1)},\nonumber\\
  \nu_{\mathrm{res}}&=(1-\alpha)\delta_1+\alpha\operatorname{Poi}(1),
    \qquad \alpha\in[0,1],
  \label{eq:classification}
\end{align}
with $\mathbb{E}_\pi U=\alpha$, $p_{\mathrm{res}}=1-\alpha/e$, and
second-moment cost $c\,a_1^2\alpha$.

Every $\alpha\in[0,1]$ is realizable by a sequence satisfying both (A4) and
(A4R). The value $\alpha=0$ is attained by exactly uniform weights. The value
$\alpha=1$ is attained by the bounded two-level endpoint construction
$\lambda_i=1-\frac1N$ for $i<N$ and $\lambda_N=2-\frac1N$, which has mean
exactly one and satisfies $\max_i\lambda_i<2$; the phenomenon is therefore
driven by the integer boundary and not by heavy tails. Each
$\alpha\in(0,1)$ is attained by placing a fraction $\alpha_N\to\alpha$ of
particles at $1-\varepsilon_N$ and the remainder at
$1+\frac{\alpha_N}{1-\alpha_N}\varepsilon_N$, with mean exactly one,
$\alpha_NN$ integral and $\varepsilon_N\downarrow0$. In the absence of
convergence of the boundary proportion, distinct subsequences may yield
distinct values of $\alpha$. Weight sequences converging to $\delta_1$ in law therefore admit residual count limits ranging from the deterministic identity to the full multinomial count law: vanishing weight heterogeneity does not imply vanishing residual
randomization.
\end{theo}
\section{The Limiting Spectral Distribution under Multinomial Resampling}
\label{sec:lsd}

\begin{theo}[Limiting spectral distribution; identity population]
\label{thm:lsd-id}
Assume multinomial resampling, (A3)--(A4), $\Sigma_N=I_d$, and the entry
convention of Section~\ref{sec:model} with finite fourth moment. Then the
empirical spectral law of $S_{r,N}$ converges in probability to
$\mu_Q=\mathcal{T}_c(\nu)$ with $\nu=\operatorname{MixPoi}(\rho)$.
\end{theo}

The supplementary material details the transfer of the classical theorem to
the random, unbounded count diagonal: truncation to $\nu^{(M)}$, a
subsequence-and-conditioning bridge for the random $D_N$, and removal of the
truncation by rank control together with continuity of the fixed point.

\begin{theo}[General population spectrum; Gaussian case]
\label{thm:lsd-gen}
Assume multinomial resampling, (A3)--(A5), and a standard Gaussian matrix
$Z_N$ independent of the weights and of the auxiliary resampling randomness.
Then the limiting spectral distribution of
$S_{r,N}=\Sigma_N^{1/2}\big(\frac1NZ_ND_NZ_N^\top\big)\Sigma_N^{1/2}$ is
\begin{equation}
  \mu_S = H\boxtimes\mu_Q = H\boxtimes\mathcal{T}_c(\nu).
  \label{eq:lsd-gen}
\end{equation}
\end{theo}

Gaussianity is used at exactly one point. Conditionally on $D_N$, which is
independent of $Z_N$ by exogeneity, the matrix
$Q_N=\frac1NZ_ND_NZ_N^\top$ is orthogonally invariant and hence
asymptotically free from the deterministic $\Sigma_N$. Unbounded counts are
handled by truncation and rank approximation before passing to the limit. No
universality claim beyond the Gaussian case is made.

\begin{rem}[Scope and normalization]
\label{rem:normalization}
Theorem~\ref{thm:lsd-id} applies verbatim to $S_{w,N}$, giving
$\mu_{\mathrm{pre}}=\mathcal{T}_c(\rho)$. As a normalization check,
$\mathcal{T}_c(\delta_1)=\mathrm{MP}(c)$, which corresponds to deterministic
counts $n_i\equiv1$ and not to multinomial resampling; uniform weights under
multinomial resampling give $\rho=\delta_1$ but $\nu=\operatorname{Poi}(1)$,
and hence a nontrivial distortion. For $c>1$, the bound
$\operatorname{rank}(S_{r,N})\le N<d$ forces a zero atom of mass at least
$1-1/c$ in $\mathcal{T}_c(\nu)$, which preserves nonnegativity in the
companion identity.
\end{rem}

\begin{prop}[Zero-eigenvalue mass and rank loss\cite{belinschi-atoms}]
\label{prop:atom}
Under the assumptions of Theorem~\ref{thm:lsd-id},
\begin{equation}
  \mu_Q(\{0\})=\max\Big(0,\ 1-\tfrac{1-\nu(\{0\})}{c}\Big),
  \label{eq:atom}
\end{equation}
and the finite-sample rank satisfies
$\operatorname{rank}(S_{r,N})\le\min(d,\#\{i:n_i>0\})$, with equality almost
surely for continuously distributed $Z$ in the identity-population case. In
the separable case of Theorem~\ref{thm:lsd-gen},
$\mu_S(\{0\})=\max\{H(\{0\}),\mu_Q(\{0\})\}$ by the maximum-atom rule for free
multiplicative convolution.
\end{prop}

Exact finite-sample rank loss and the zero mass of the limiting spectral
distribution are distinct notions and are established separately. We note that
$\operatorname{rank}(\Sigma_N)/d\to1-H(\{0\})$ need not hold under (A5) alone.

\begin{prop}[Spectral moments and the exact increment]
\label{prop:moments}
Under the assumptions of Theorem~\ref{thm:lsd-gen},
\begin{equation}
  m_{2,\mathrm{pre}}=m_{2,H}+c\,a_1^2\,r_2,\quad
  m_{2,\mathrm{post}}=m_{2,H}+c\,a_1^2\,(1+r_2),
  \label{eq:moments}
\end{equation}
so that the multinomial increment equals $c\,a_1^2$. The first moment is
preserved in the same limiting sense, with $m_1=a_1$ on both sides. The
normalized empirical first and second moments converge in probability, and not
merely in expectation.
\end{prop}

\begin{prop}[Information content of the effective sample size]
\label{prop:ess}
The effective sample size satisfies $N/\mathrm{ESS}\to r_2$, equivalently
$\mathrm{ESS}/N\to1/r_2$. Given $c$ and $H$, ESS therefore determines the
multinomial limiting second spectral moment, but it does not determine the
full limiting spectral distribution. It also yields an occupancy bound: for
$K\sim\operatorname{MixPoi}(\rho)$ we have $\mathbb{E}K=1$ and
$\mathbb{E}K^2=1+r_2$, so that
$1=(\mathbb{E}K)^2\le\mathbb{E}K^2\cdot\mathbb{P}(K>0)$, giving
\begin{equation}
  p_{\mathrm{multi}}\ \ge\ \tfrac{1}{1+r_2}.
  \label{eq:occ-bound}
\end{equation}
If $r_2=1$ with mean $1$, then necessarily $\rho=\delta_1$ and the
multinomial count law is determined, with $\nu=\operatorname{Poi}(1)$.

In general, however, ESS determines neither the occupancy structure, nor the
zero mass, nor the full spectrum, nor the residual post-resampling second
moment, the last of which requires $\mathbb{E}_\pi U$. Explicitly, the two
weight laws
\begin{equation}
  \rho^{(A)}=\tfrac12\delta_{1/4}+\tfrac12\delta_{7/4},\quad
  \rho^{(B)}=\tfrac9{10}\delta_{3/4}+\tfrac1{10}\delta_{13/4}
  \label{eq:ess-pair}
\end{equation}
share mean $1$ and second moment $25/16$, and hence share the value
$\mathrm{ESS}=16N/25$ exactly at multiples of $10$ by a deterministic ratio
construction, yet have count zero masses $0.476287$ and $0.429007$ and count
third moments $8.375$ and $9.5$. Since
$m_3(\mathcal{T}_c(\nu))=1+3cb_2+c^2b_3$ with equal $b_2$ and distinct $b_3$,
the two limiting spectra differ for every $c>0$; at $c=1$ the
zero-eigenvalue masses are $0.476287$ and $0.429007$, since then
$\mu_Q(\{0\})=\nu(\{0\})$.
\end{prop}

\begin{coro}[Asymptotic decrease of the effective dimension]
\label{coro:deff}
Under the assumptions of Proposition~\ref{prop:moments} with $a_1>0$ and
$c>0$, set $d_{\mathrm{eff}}(M):=(\operatorname{tr}M)^2/\operatorname{tr}M^2$.
Then
\begin{align}
  \tfrac{d_{\mathrm{eff}}(S_{w,N})}{d}
    &\xrightarrow{\mathbb{P}} \tfrac{a_1^2}{m_{2,H}+c\,a_1^2r_2},\nonumber\\
  \tfrac{d_{\mathrm{eff}}(S_{r,N})}{d}
    &\xrightarrow{\mathbb{P}}
      \tfrac{a_1^2}{m_{2,H}+c\,a_1^2(1+r_2)},
  \label{eq:deff}
\end{align}
with explicit gap
$c\,a_1^4/[(m_{2,H}+c\,a_1^2r_2)(m_{2,H}+c\,a_1^2(1+r_2))]>0$.
\end{coro}

At the finite-sample level we assert only that the unnormalized
$\operatorname{tr}S_{r,N}^2$ is nondecreasing in conditional expectation. The
conditional expectation of the inverse participation ratio
$\operatorname{tr}M^2/(\operatorname{tr}M)^2$ can decrease strictly already for
$d=N=2$.

\section{Rank Loss and the Spectral-Gap Transition}
\label{sec:rankgap}

Let $p_{\mathsf{s}}:=1-\nu_{\mathsf{s}}(\{0\})\in(0,1]$ denote the limiting
surviving-parent fraction, so that
$p_{\mathrm{multi}}=1-\mathcal{L}_\rho(1)$ for multinomial resampling.

\begin{theo}[Rank-and-spectral-gap transition\cite{baik2005phase}]
\label{thm:rankgap}
Assume multinomial resampling, (A3)--(A4), $\Sigma_N=I_d$, and the entry
convention of Section~\ref{sec:model} with finite fourth moment. The active
index set is independent of the particle matrix by exogeneity, so that
conditionally on the counts the active columns remain independent and
identically distributed with the original entry distribution. Let
$M_N:=\#\{i:n_i>0\}$, so that $M_N/N\to p$ in probability. Then the following
hold.
\begin{itemize}
  \item[(i)] \emph{Subcritical regime} $c<p$:
    $\mathbb{P}(\operatorname{rank}(S_{r,N})=d)\to1$, and for every
    $\varepsilon>0$,
    \begin{equation}
      \mathbb{P}\Big(\lambda_{\min}(S_{r,N})
        \ge(\sqrt p-\sqrt c\,)^2-\varepsilon\Big)\to1.
      \label{eq:subcrit}
    \end{equation}
  \item[(ii)] \emph{Supercritical regime} $c>p$: $\mu_Q(\{0\})=1-p/c$, and for
    every $\varepsilon>0$,
    \begin{equation}
      \mathbb{P}\Big(\lambda_{\min}^{+}(S_{r,N})
        \ge(\sqrt c-\sqrt p\,)^2-\varepsilon\Big)\to1.
      \label{eq:supercrit}
    \end{equation}
  \item[(iii)] \emph{Criticality} $c=p$: $\mu_Q(\{0\})=0$ and
    $\inf\operatorname{supp}\mu_Q=0$, so that the limiting spectral support
    touches zero without carrying an atom there, and
    \begin{equation}
      \lambda_{\min}^{+}(S_{r,N})\xrightarrow{\ \mathbb{P}\ }0.
      \label{eq:crit}
    \end{equation}
\end{itemize}
\end{theo}

Two clarifications are in order. First, the bounds in (i) and (ii) constitute
an explicit universal lower bound, a Marchenko--Pastur lower-edge benchmark
depending only on $p$. The true lower edge depends on the full count law and
may be strictly larger than this benchmark; it is characterized by
Proposition~\ref{prop:edges}, and it is strictly larger at every off-critical
grid point examined in Section~\ref{sec:numerics}. Second, for general
$\Sigma_N$ no uniform lower gap can hold under (A5) alone; a gap statement in
the separable case requires in addition that
$\lambda_{\min}(\Sigma_N)\ge a_->0$.

\begin{prop}[Spectral edges and support\cite{benaych2009rectangular,bai1993limit}]
\label{prop:edges}
Under the assumptions of Theorem~\ref{thm:lsd-id}, define
\begin{equation}
  z(u):=-\tfrac1u+y\int\tfrac{t}{1+tu}\,\nu(dt),\qquad y=\tfrac1c,
  \label{eq:edge-map}
\end{equation}
on real $u\neq0$ with $-1/u\notin\operatorname{supp}\nu$. The
support-complement parametrization uses the real branches on which $z'(u)>0$;
regular edges of the nonzero part of $\mu_Q$ arise at boundary critical points
of these branches and take the form $c\,z(u^*)$. Not every real root of
$z'(u)=0$ corresponds to an edge, and the critical zero endpoint, which
corresponds to $u\to\infty$ rather than to a finite critical point, is treated
separately in Theorem~\ref{thm:rankgap}(iii). As a consistency check,
$\nu=\delta_1$ gives $c(1\mp\sqrt{1/c})^2=(1\mp\sqrt c)^2$.

Since $\nu=\operatorname{MixPoi}(\rho)$ charges every nonnegative integer
whenever $\rho((0,\infty))>0$, the limiting spectrum is unbounded on the
right:
\begin{equation}
  \sup\operatorname{supp}\mathcal{T}_c\big(\operatorname{MixPoi}(\rho)\big)
    =\infty,
  \label{eq:unbounded}
\end{equation}
and consequently $\lambda_{\max}(S_{r,N})\to\infty$ in probability.
\end{prop}

The proof of \eqref{eq:unbounded} exhibits a positive limiting fraction of
eigenvalues above any prescribed level; a single escaping outlier would not
suffice. There is thus no finite global right edge, the lower and internal
edges remain characterized as above, and bulk extent must be treated as a
quantile notion. Whether the support can separate into several intervals for
particular $\rho$ is left open.
\section{Residual Resampling}
\label{sec:res}

\begin{theo}[Residual spectral transform]
\label{thm:res-lsd}
Under (A3), (A4R), identity population covariance, and the entry assumptions
of Theorem~\ref{thm:lsd-id}, the empirical spectral law of $S_{r,N}$ under
residual resampling converges in probability to
\begin{equation}
  \mu_Q^{\mathrm{res}} = \mathcal{T}_c(\nu_{\mathrm{res}}).
  \label{eq:res-lsd}
\end{equation}
If instead the population covariance satisfies (A5) and $Z_N$ is Gaussian,
under the same exogeneity and independence assumptions,
$\mu_S^{\mathrm{res}}=H\boxtimes\mathcal{T}_c(\nu_{\mathrm{res}})$.
\end{theo}

The transfer of the empirical second-moment convergence to the residual
setting is established separately in the supplementary material and is not
inherited from the statement on the limiting spectral distribution.

\begin{coro}[Residual zero atom and threshold]
\label{coro:res-atom}
We have
\begin{equation}
  \begin{aligned}
    \nu_{\mathrm{res}}(\{0\})
      &= \mathbb{E}_\pi\big[\mathbf{1}_{\{K=0\}}\,e^{-U}\big],\\
    p_{\mathrm{res}}(\pi) &:= 1-\nu_{\mathrm{res}}(\{0\}).
  \end{aligned}
  \label{eq:res-atom}
\end{equation}
Under the identity-population and entry assumptions of
Theorem~\ref{thm:rankgap}, with (A4R) in place of (A4), all three regimes of
that theorem hold with $p_{\mathrm{res}}$ in place of $p$. The noncritical
case transfers through active-column domination and the critical case through
the residual limiting spectral distribution and its fixed-point equation.
\end{coro}

\begin{prop}[Residual second-moment cost]
\label{prop:res-cost}
Under (A4R) we have $\mathbb{E}[K+P_U\mid K,U]=K+U$ and
$\operatorname{Var}(K+P_U\mid K,U)=U$, whence
$b_2(\nu_{\mathrm{res}})=r_2+\mathbb{E}_\pi U$. Under the residual analogue of
the assumptions of Proposition~\ref{prop:moments},
\begin{equation}
  m_{2,\mathrm{res}}-m_{2,\mathrm{pre}}
    = c\,a_1^2\,\mathbb{E}_\pi U\ \le\ c\,a_1^2,
  \label{eq:res-cost}
\end{equation}
with strict inequality if and only if $a_1>0$ and $\mathbb{E}_\pi U<1$. The
comparison is one of asymptotic normalized second spectral moments; no
ordering for arbitrary spectral functionals is claimed.
\end{prop}

The cost vanishes if and only if $a_1=0$ or $\mathbb{E}_\pi U=0$; assuming
$a_1>0$, if and only if $\mathbb{E}_\pi U=0$, that is, if and only if $\pi$ is
supported on $\{U=0\}$. This is a property of the limit law and not of every
finite sample. For $a_1>0$, equality with the multinomial asymptotic cost
occurs precisely at the integer-boundary limit $\pi=\delta_{(0,1)}$, since
$\mathbb{E}_\pi(K+U)=1$ together with $\mathbb{E}_\pi U=1$ forces $U=1$ and
$K=0$ almost surely under $\pi$. Exact zero cost at finite $N$ holds whenever
the counts are deterministic, in particular when every
$\lambda_i\in\mathbb{Z}_{\ge0}$, and under the geometric degeneracy of
Remark~\ref{rem:vanish}.

\begin{coro}[Asymptotic rank-preservation ordering]
\label{coro:rank-order}
For any joint limit $\pi$,
\begin{equation}
  \nu_{\mathrm{multi}}(\{0\})-\nu_{\mathrm{res}}(\{0\})
    = \mathbb{E}_\pi\big[\mathbf{1}_{\{K\ge1\}}\,e^{-(K+U)}\big]\ \ge\ 0,
  \label{eq:rank-order}
\end{equation}
so that $\mu_{\mathrm{res}}(\{0\})\le\mu_{\mathrm{multi}}(\{0\})$: residual
resampling preserves asymptotically at least as much rank as multinomial
resampling. The statement orders limiting occupancy and the zero mass of the
limiting spectral distribution; it is not a sample-path rank comparison.

The count-level inequality is strict if and only if $\pi(K\ge1)>0$. At the
spectral level, writing $q_{\mathrm{multi}}:=(1-p_{\mathrm{multi}}/c)_+$ and
$h_0:=H(\{0\})$, we have
$\mu_{\mathrm{res}}(\{0\})<\mu_{\mathrm{multi}}(\{0\})$ if and only if
$\pi(K\ge1)>0$ and $q_{\mathrm{multi}}>h_0$; in the identity-population case
$h_0=0$ and the condition reduces to $c>p_{\mathrm{multi}}$. On the window
$c\in(p_{\mathrm{multi}},p_{\mathrm{res}})$, multinomial resampling already
carries a macroscopic zero atom while residual resampling has full rank with
probability tending to one.
\end{coro}

\begin{coro}[Exact-uniform-weight contrast]
\label{coro:uniform-contrast}
Let $w_i=1/N$ for every $N$. Residual resampling then leaves the second-moment
matrix unchanged at every finite $N$, so that in the identity-population
setting its limiting spectral distribution is $\mathrm{MP}(c)$, with zero
resampling cost and $p_{\mathrm{res}}=1$; in a general Gaussian population
model the unchanged matrix has limiting spectral distribution
$H\boxtimes\mathrm{MP}(c)$. Multinomial resampling produces
$\nu=\operatorname{Poi}(1)$, cost $c\,a_1^2$, and the transition at
$p_{\mathrm{multi}}=1-e^{-1}$. As classified in
Theorem~\ref{thm:classification}, sequences satisfying the corresponding
joint-limit and moment assumptions realize every intermediate residual limit
$(1-\alpha)\delta_1+\alpha\operatorname{Poi}(1)$ with $\alpha\in[0,1]$.
\end{coro}

\begin{coro}[Finite-sample Frobenius ordering]
\label{coro:frob-order}
Under (A1)--(A2), set $k_i=\lfloor\lambda_i\rfloor$,
$u_i=\operatorname{frac}(\lambda_i)$, $L=\sum_ik_i$ and $R=\sum_iu_i$. When
$L,R>0$,
\begin{align}
  \operatorname{Cov}_{\mathrm{multi}}-\operatorname{Cov}_{\mathrm{res}}
    =\ &\operatorname{diag}(k)-\tfrac{kk^\top}{L}\nonumber\\
    &+\tfrac{LR}{N}\big(\tfrac kL-\tfrac uR\big)
      \big(\tfrac kL-\tfrac uR\big)^\top\ \succeq\ 0.
  \label{eq:frob-order}
\end{align}
The boundary cases are immediate: if $R=0$ then
$\operatorname{Cov}_{\mathrm{res}}=0$, and if $L=0$ the two covariances
coincide. Hence, by the quadratic form of Proposition~\ref{prop:cost} with
$G\circ G$ positive semidefinite,
\begin{equation}
  0\le\mathbb{E}\big[\|S_{\mathrm{res}}-S_w\|_F^2\mid X,w\big]
    \le\mathbb{E}\big[\|S_{\mathrm{multi}}-S_w\|_F^2\mid X,w\big].
  \label{eq:frob-ineq}
\end{equation}
\end{coro}

Corollary~\ref{coro:frob-order} is a count-covariance-level comparison in the
error energy associated with the second spectral moment cost of this paper. No
ordering for arbitrary convex spectral functionals is claimed, and the result
is not presented as a general variance-comparison theory for resampling
schemes.

\section{Numerical Study}
\label{sec:numerics}

The numerical study is organized by the logical layers of the theory rather
than by experiment size: the count laws of Section~\ref{sec:count} are
examined first, then the spectral map of Definition~\ref{def:Tc}, and finally
the consequences established in Sections~\ref{sec:lsd}--\ref{sec:res}. Every
quantity reported below is compared against a closed-form prediction fixed
prior to the corresponding simulation, so that each comparison admits a
pass-or-fail reading rather than a qualitative one.

\subsection{Design and validation protocol}
\label{sec:ix-design}

For each experiment the target quantities, sample sizes, repetition counts
and acceptance criteria were fixed in advance; the theoretical values are
evaluated in closed form from the statements of
Sections~\ref{sec:count}--\ref{sec:res} and are not fitted. All runs use a
single fixed seed, no reported result is selected from a larger pool of runs,
and each figure addresses a single question. Within each repetition the
schemes are compared on paired particle clouds, so that differences between
schemes are not confounded with variation between clouds, and confidence
intervals are computed across independent repetitions and reported as one
standard error of the mean. Atoms are plotted as separate point masses and are
never merged into a density. Since Theorem~\ref{thm:classification} depends on
the manner in which the weights approach the integers, integer parts are
stored as integers and $\lambda_i$ is never reconstructed from a
floating-point product $Nw_i$, and the endpoint construction uses $N=2^m$ so
that $1-1/N$ and $2-1/N$ are exactly representable.

\emph{Discrepancy metric.} Agreement between a simulated spectrum and a
predicted law is measured by a fixed-bandwidth Stieltjes discrepancy
\begin{equation}
  E_h := \big| m_{\mathrm{emp}}(x+\mathrm{i}h)
    - m_{\mathrm{th}}(x+\mathrm{i}h) \big|,
  \label{eq:Eh}
\end{equation}
evaluated on a stated grid and reported in both maximum and grid-averaged
form, together with a sensitivity check at bandwidth $h/2$. This choice avoids
the bandwidth arbitrariness of density estimation while remaining sensitive to
atoms and to edges. All spectral predictions are computed from the fixed-point
equation of Definition~\ref{def:Tc}. Since several classically equivalent
companion-side forms coincide at $\eta=\delta_1$ while governing a different
object in the present setting, recovery of the Marchenko--Pastur law alone
does not validate an implementation; the solver therefore passes a battery of
unit tests before any experiment is run, each reported with its realized
tolerance in the supplementary material~(S.23).
\subsection{The count-law-to-spectrum map}
\label{sec:ix-e1}

\emph{Design.} The configuration uses an identity population, Gaussian
entries, $N=1000$, $d=800$ so that $c=0.8$, and $n_{\mathrm{rep}}=20$. Two
weight profiles are considered: exactly uniform weights, and the two-point law
$\rho=\tfrac12\delta_{1/2}+\tfrac12\delta_{3/2}$ realized by an exact ratio
construction of mean exactly one. For each profile the empirical count
probability mass function is compared against $\operatorname{MixPoi}(\rho)$ and
$\nu_{\mathrm{res}}$, and the simulated spectra of the pre-resampling,
multinomial and residual matrices are compared against their $\mathcal{T}_c$
predictions, with atoms drawn separately.

\emph{Predictions.} All targets follow in closed form. With $H=\delta_1$,
$a_1=1$, $m_{2,H}=1$ and $c=0.8$, Propositions~\ref{prop:moments}
and~\ref{prop:res-cost} give

\begin{equation}
    \begin{aligned}
      m_2^{\mathrm{pre}} &= 1 + c\,r_2, \qquad
      m_2^{\mathrm{multi}} = 1 + c(1+r_2), \\[2pt]
      m_2^{\mathrm{res}} &= m_2^{\mathrm{pre}} + c\,\mathbb{E}_\pi U,
    \end{aligned}
    \label{eq:e1pred}
  \end{equation}
while Theorems~\ref{thm:multi-count} and~\ref{thm:res-count} give the count
zero masses $\mathcal{L}_\rho(1)$ and
$\mathbb{E}_\pi[\mathbf{1}_{\{K=0\}}e^{-U}]$, and
Proposition~\ref{prop:atom} converts each into a spectral atom. For the
uniform profile $r_2=1$ and $\mathbb{E}_\pi U=0$; for the two-point profile
$r_2=5/4$ and $\mathbb{E}_\pi U=1/2$, since both atoms are located at
fractional part $1/2$.

\begin{figure*}[!t]
  \centering
  \subfloat[pre-resampling]{\includegraphics{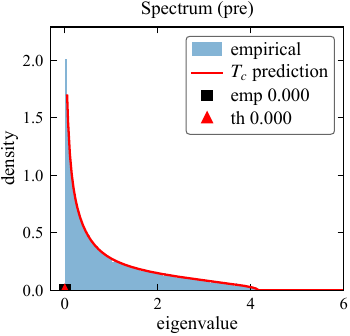}%
    \label{fig:e1a}}\hfil
  \subfloat[multinomial]{%
    \includegraphics{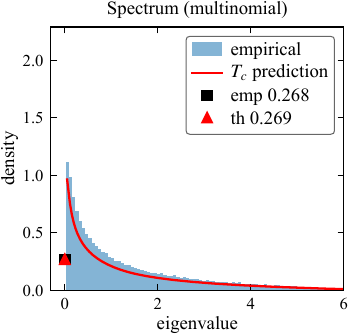}\label{fig:e1b}}\hfil
  \subfloat[residual]{%
    \includegraphics{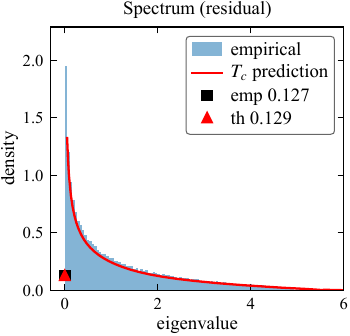}\label{fig:e1c}}
  \caption{One map, three count laws: simulated spectra against
    $\mathcal{T}_c$ predictions under the two-point weight law
    $\rho=\tfrac12\delta_{1/2}+\tfrac12\delta_{3/2}$, with an identity
    population and Gaussian entries ($N=1000$, $d=800$, $c=0.8$,
    $n_{\mathrm{rep}}=20$). Zero atoms are drawn as separate point masses at
    the origin and are not merged into the density; all three panels share
    both axes. The three schemes separate as predicted, with second moments
    $2.000$, $2.800$ and $2.400$. The exactly uniform profile is omitted,
    since there residual resampling reproduces the pre-resampling matrix
    itself and the two panels would coincide exactly, which is the content of
    Corollary~\ref{coro:uniform-contrast}; see Table~\ref{tab:e1}. The
    corresponding count probability mass functions are reported in the
    supplementary material~(S.25).}
  \label{fig:e1}
\end{figure*}

\begin{table*}[!t]
\renewcommand{\arraystretch}{1.25}
\caption{Experiment E1: pre-registered spectral quantities against theory.
  Empirical values are means over $n_{\mathrm{rep}}=20$ independent
  repetitions, reported with one standard error. Here $\nu(\{0\})$ denotes the
  count zero mass and $\mu_Q(\{0\})$ the induced spectral atom. Under the
  uniform profile the residual row reproduces the pre-resampling row exactly,
  and not merely to within the reported error. The first spectral moment is
  $1.000$ to within Monte Carlo error in every row and is omitted.}
\label{tab:e1}
\centering
\begin{tabular}{llcccccc}
\toprule
& & \multicolumn{2}{c}{$m_2$}
  & \multicolumn{2}{c}{$\nu(\{0\})$} & \multicolumn{2}{c}{$\mu_Q(\{0\})$} \\
\cmidrule(lr){3-4}\cmidrule(lr){5-6}\cmidrule(lr){7-8}
Profile & Scheme & empirical & theory & empirical & theory
        & empirical & theory \\
\midrule
uniform & pre         & $1.8039\pm0.0012$ & $1.800$
                      & $0$ & $0$ & $0$ & $0$ \\
uniform & multinomial & $2.6118\pm0.0074$ & $2.600$
                      & $0.36805\pm0.00195$ & $0.367879$
                      & $0.21006\pm0.00244$ & $0.209849$ \\
uniform & residual    & $1.8039\pm0.0012$ & $1.800$
                      & $0$ & $0$ & $0$ & $0$ \\
\addlinespace[2pt]
two-point & pre         & $2.0011\pm0.0019$ & $2.000$
                        & $0$ & $0$ & $0$ & $0$ \\
two-point & multinomial & $2.7985\pm0.0103$ & $2.800$
                        & $0.41215\pm0.00230$ & $0.414830$
                        & $0.26519\pm0.00288$ & $0.268538$ \\
two-point & residual    & $2.4026\pm0.0057$ & $2.400$
                        & $0.30580\pm0.00148$ & $0.303265$
                        & $0.13225\pm0.00184$ & $0.129081$ \\
\bottomrule
\end{tabular}
\end{table*}

\emph{Results.} Every pre-registered quantity agrees to within $2\sigma$, and
Table~\ref{tab:e1} reports the full comparison. Two features of the comparison
are not apparent from the tabulated values alone. First, under exactly uniform
weights the residual row does not merely agree with the pre-resampling row to
within the reported error: it reproduces that row exactly, since the realized
matrix is the same, as required by Corollary~\ref{coro:uniform-contrast} at
every finite $N$ and not only in the limit. Second, the Stieltjes
discrepancies at $h=0.05$ have grid means between $0.0019$ and $0.0040$ and
maxima of at most $0.0302$, localized in the left boundary layer; halving the
bandwidth raises the maxima to at most $0.0428$ without altering any
conclusion, which identifies the residual misfit as a boundary-layer effect
rather than a bulk one.

Under the two-point profile the three schemes separate as predicted, with
second moments $2.000$, $2.800$ and $2.400$. Under exactly uniform weights,
multinomial resampling nevertheless produces a zero atom of mass $0.21$. This
distortion is not a small-sample artifact: it is the value predicted by
$\mathcal{T}_c(\operatorname{Poi}(1))$.

The same framework predicts the rank-and-spectral-gap transition at
$c=p_{\mathsf{s}}$ and the conservativeness of the Marchenko--Pastur benchmark
edge, which costs a factor of approximately $1.3$ in the gap it certifies;
both are verified over a twelve-point sweep in $c$ and are reported in the
supplementary material~(S.24).

\subsection{Identical effective sample size with different null spaces}
\label{sec:ix-e3}

Proposition~\ref{prop:ess} states that ESS determines the limiting second
spectral moment given $c$ and $H$, but does not determine the spectrum. The
present experiment quantifies the resulting gap, which is of practical
relevance because ESS is the standard scalar diagnostic for resampling.

\emph{Design.} The pair
\begin{equation}
  \rho^{(A)} = \tfrac12\delta_{1/4}+\tfrac12\delta_{7/4},
  \qquad
  \rho^{(B)} = \tfrac{9}{10}\delta_{3/4}+\tfrac{1}{10}\delta_{13/4}
  \label{eq:e3pair}
\end{equation}
of Proposition~\ref{prop:ess} shares mean $1$ and second moment $25/16$, and
hence shares the value $\mathrm{ESS}=16N/25$ whenever $N$ is a multiple of
$10$, while having distinct count third moments $8.375$ and $9.5$. The
configuration uses multinomial resampling with $c=0.8$, $N=1000$ and
$n_{\mathrm{rep}}=30$. Four quantities are reported: ESS and the predicted
$m_2$, which coincide by construction; the count zero masses; and the spectral
atoms. The count third moments enter through $m_3=1+3cb_2+c^2b_3$.

\begin{figure}[!t]
  \centering
  \includegraphics{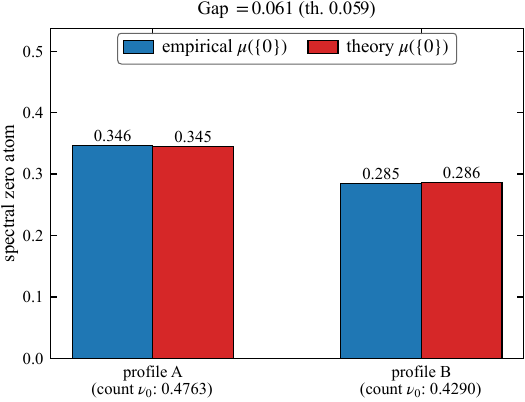}
  \caption{Two weight profiles with exactly equal ESS, of value $640$, and
    statistically indistinguishable second spectral moments, but
    macroscopically different null spaces, under multinomial resampling with
    $c=0.8$, $N=1000$ and $n_{\mathrm{rep}}=30$. The spectral atoms differ by
    $0.061$ against a pre-registered prediction of $0.059$, a separation of
    nineteen standard errors. Any diagnostic based on ESS reports the two
    configurations as identical.}
  \label{fig:e3}
\end{figure}

\begin{table}[!t]
\renewcommand{\arraystretch}{1.25}
\caption{Experiment E3: identical ESS with different null spaces. Empirical
  values are means over $n_{\mathrm{rep}}=30$ repetitions, reported with one
  standard error. ESS is verified as an exact identity and is not estimated.}
\label{tab:e3}
\centering
\begin{tabular}{lcc}
\toprule
Quantity & Profile $A$ & Profile $B$ \\
\midrule
$\mathrm{ESS}$ (exact)   & $640$ & $640$ \\
$m_2$ empirical          & $3.0559\pm0.0101$ & $3.0397\pm0.0161$ \\
$m_2$ theory             & $3.05$ & $3.05$ \\
$\nu(\{0\})$ empirical   & $0.47717\pm0.00143$ & $0.42823\pm0.00215$ \\
$\nu(\{0\})$ theory      & $0.476287$ & $0.429007$ \\
$\mu_Q(\{0\})$ empirical & $0.34646\pm0.00179$ & $0.28529\pm0.00268$ \\
$\mu_Q(\{0\})$ theory    & $0.345359$ & $0.286259$ \\
$m_3$ empirical          & $12.576\pm0.097$ & $13.171\pm0.180$ \\
$m_3$ theory             & $12.510$ & $13.230$ \\
\bottomrule
\end{tabular}
\end{table}

\emph{Results.} ESS equals $640$ on both sides, verified as an identity
rather than estimated, and the second spectral moments are statistically
indistinguishable, with $|\Delta|=0.0162$ against a joint $2\sigma$ of
$0.0381$. The null spaces are not indistinguishable: the spectral atoms
separate by $0.0612$ against a predicted $0.0591$, or nineteen standard errors,
and the count third moments corroborate the mechanism.
Table~\ref{tab:e3} reports the full comparison.

Two weight configurations that any diagnostic based on ESS reports as identical
therefore differ by six percentage points of the ambient dimension in the
fraction of directions annihilated by a single resampling step. The quantity
involved is finite and directly measurable rather than asymptotic.

\subsection{Residual randomization at the integer boundary}
\label{sec:ix-e4}

\emph{Design.} The three constructions of Theorem~\ref{thm:classification} are
used, all satisfying $\rho_N\Rightarrow\delta_1$: exactly uniform weights,
giving $\alpha=0$; the interior family at $\alpha=0.4$, in which a fraction
$\alpha_N\to\alpha$ of particles is placed at $1-\varepsilon_N$ and the
remainder at $1+\frac{\alpha_N}{1-\alpha_N}\varepsilon_N$, so that
\begin{equation}
  \frac1N\sum_i \lambda_i^2
    = 1 + \frac{\alpha_N}{1-\alpha_N}\varepsilon_N^2,
  \qquad
  \frac{R_N}{N} = \alpha_N;
  \label{eq:e4interior}
\end{equation}
and the bounded two-level endpoint construction $\lambda_i=1-\frac1N$ for
$i<N$ with $\lambda_N=2-\frac1N$ and $N=2^m$, giving $\alpha=1$. Each
construction is run at $N=2^8,\dots,2^{14}$. The reported quantities are
count-level and therefore inexpensive to evaluate: $R_N/N$, the zero-count
fraction, and the second count moment, each compared against the predictions
$\alpha$, $\alpha e^{-1}$ and $1+\alpha$ of
Theorem~\ref{thm:classification} and Proposition~\ref{prop:res-cost}, and
displayed alongside the second-order weight deviation
$\frac1N\sum_i(\lambda_i-1)^2$.

\begin{figure*}[!t]
  \centering
  \subfloat[Second-order weight deviation]{%
    \includegraphics{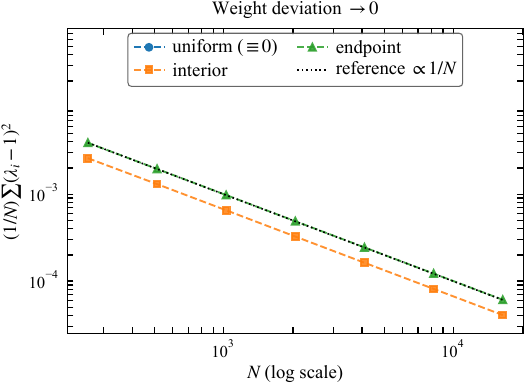}\label{fig:e4a}}\hfil
  \subfloat[Residual randomization fraction]{%
    \includegraphics{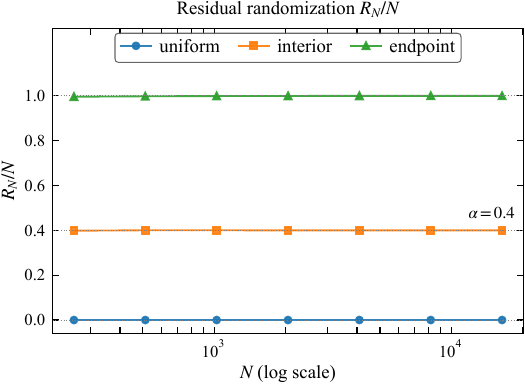}\label{fig:e4b}}
  \caption{Vanishing weight heterogeneity does not imply vanishing residual
    randomization, in accordance with Theorem~\ref{thm:classification}.
    \protect\subref{fig:e4a} The second-order deviation
    $\frac1N\sum_i(\lambda_i-1)^2$ vanishes at rate $1/N$ in all three
    constructions, with the dotted line giving the reference slope. The uniform
    construction is omitted from this panel and marked as identically zero in
    the legend, since its deviation vanishes exactly at every $N$ and cannot be
    displayed on a logarithmic axis.
    \protect\subref{fig:e4b} Over the same range, $R_N/N$ converges to the
    three distinct limits $0$, $0.4$ and $1$. The two panels are to be read
    together: the simultaneous vanishing of weight heterogeneity and
    persistence of randomization constitutes the content of
    Theorem~\ref{thm:classification}. Since the endpoint construction keeps
    all weights below $2$, the effect originates in proximity to the integer
    grid rather than in heavy tails.}
  \label{fig:e4}
\end{figure*}

\begin{table*}[!t]
\renewcommand{\arraystretch}{1.25}
\caption{Experiment E4: count-level quantities for the three constructions of
  Theorem~\ref{thm:classification}, at the endpoints of the sweep; the full
  sweep $N=2^8,\dots,2^{14}$ is tabulated in the supplementary
  material~(S.26). For
  the uniform construction the counts are
  deterministic, so that all entries are exact and carry no Monte Carlo
  error.}
\label{tab:e4}
\centering
\begin{tabular}{llcccc}
\toprule
Construction & $N$ & $R_N/N$ & zero fraction & $b_{2,N}$
             & $\frac1N\sum(\lambda_i-1)^2$ \\
\midrule
uniform   & all      & $0$ & $0$ & $1$ & $0$ \\
          & (theory) & $0$ & $0$ & $1$ & --- \\
\addlinespace[2pt]
interior  & $256$    & $0.3984$ & $0.15408\pm0.00089$
                     & $1.3963\pm0.0037$ & $2.59\times10^{-3}$ \\
          & $16384$  & $0.4000$ & $0.14835\pm0.00012$
                     & $1.3998\pm0.0005$ & $4.07\times10^{-5}$ \\
          & (theory) & $0.4$    & $0.147152$ & $1.4$ & $\to0$ \\
\addlinespace[2pt]
endpoint  & $256$    & $0.9961$ & $0.36859\pm0.00135$
                     & $2.0034\pm0.0062$ & $3.89\times10^{-3}$ \\
          & $16384$  & $0.9999$ & $0.36800\pm0.00018$
                     & $2.0007\pm0.0008$ & $6.10\times10^{-5}$ \\
          & (theory) & $1$      & $0.367879$ & $2$ & $\to0$ \\
\bottomrule
\end{tabular}
\end{table*}

\emph{Results.} The second-order weight deviation vanishes at rate $1/N$ in
all three constructions: for the interior family it decreases from
$2.59\times10^{-3}$ at $N=256$ to $4.07\times10^{-5}$ at $N=16384$, a factor of
$64$ across a factor of $64$ in $N$. Over the same range $R_N/N$ converges to
the three distinct limits $0$, $0.400$ and $1$, and Table~\ref{tab:e4} reports
the count-level quantities against their predictions $\alpha$, $\alpha e^{-1}$
and $1+\alpha$.

Weight heterogeneity, as measured by any moment of $\rho_N-\delta_1$, may
therefore vanish while residual resampling retains the full randomization of
the multinomial scheme. Since the endpoint construction keeps all weights below
$2$, the effect is governed by the position of the weights relative to the
integer grid and not by heavy tails, and near-uniformity of the weights does
not permit the inference that residual resampling is nearly deterministic.

\subsection{The residual atom-free window}
\label{sec:ix-e4spec}

\emph{Design.} A single spectral configuration of moderate size is used, with
$N=1024$, $d=819$ so that $c=0.7998$, and the interior construction at
$\alpha=0.5$. The pre-resampling, multinomial and residual spectra are
compared against their $\mathcal{T}_c$ predictions. The value $\alpha=0.5$ is
used here in place of the value $\alpha=0.4$ of Fig.~\ref{fig:e4} because it
places $c$ within the window of Corollary~\ref{coro:rank-order} with a wider
margin; the two values yield the same qualitative conclusion.

\emph{Prediction.} At $\alpha=0.5$ the residual count law is
$\tfrac12\delta_1+\tfrac12\operatorname{Poi}(1)$, so that
\begin{equation}
  p_{\mathrm{res}} = 1-\frac{\alpha}{e} = 0.8161,
  \qquad
  p_{\mathrm{multi}} = 1-e^{-1} = 0.6321 .
  \label{eq:e4window}
\end{equation}
With $c=0.7998$ the configuration lies within the window
$c\in(p_{\mathrm{multi}},p_{\mathrm{res}})$ of
Corollary~\ref{coro:rank-order}. The multinomial spectrum must therefore carry
an atom of mass $1-p_{\mathrm{multi}}/c=0.210$, the residual spectrum must
carry none, and the pre-resampling spectrum is $\mathrm{MP}(c)$ without an
atom, since $c<1$.

\begin{figure*}[!t]
  \centering
  \subfloat[pre-resampling]{%
    \includegraphics{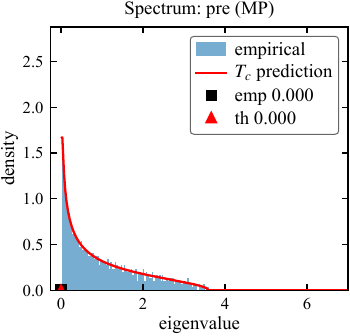}\label{fig:e4sa}}\hfil
  \subfloat[multinomial]{%
    \includegraphics{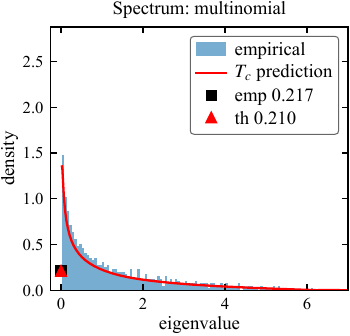}\label{fig:e4sb}}\hfil
  \subfloat[residual]{%
    \includegraphics{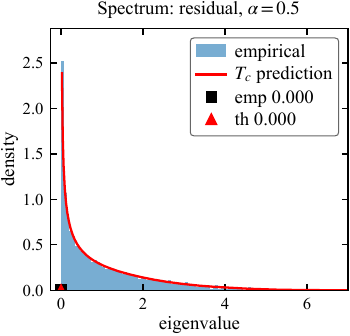}\label{fig:e4sc}}
  \caption{The atom-free window of Corollary~\ref{coro:rank-order}, at
    $N=1024$, $d=819$ and $c=0.7998$, with the interior construction at
    $\alpha=0.5$. The value of $\alpha$ differs from that of
    Fig.~\ref{fig:e4}, where $\alpha=0.4$; the larger value places $c$ near
    the centre of the window
    $(p_{\mathrm{multi}},p_{\mathrm{res}})=(0.632,0.816)$. All three panels
    share both axes. \protect\subref{fig:e4sa} The pre-resampling spectrum
    recovers $\mathrm{MP}(c)$. \protect\subref{fig:e4sb} Multinomial
    resampling carries a zero atom of empirical mass $0.217$ against the
    predicted $0.210$. \protect\subref{fig:e4sc} Residual resampling carries
    no atom: within this window it retains full rank with probability tending
    to one, whereas multinomial resampling has already annihilated a fixed
    fraction of directions.}
  \label{fig:e4spec}
\end{figure*}

\emph{Results.} The three panels realize the prediction: the multinomial
spectrum carries an atom of empirical mass $0.217$ against the predicted
$0.210$, the residual spectrum carries none, and the pre-resampling spectrum
recovers $\mathrm{MP}(c)$. Within this window the two schemes differ
qualitatively rather than quantitatively, since multinomial resampling has
already eliminated a fixed fraction of directions while residual resampling
retains full rank with probability tending to one. For $\alpha=0.5$ the window
spans $c\in(0.632,0.816)$, a range that includes many practical ratios of
particle number to dimension.

\subsection{The finite-sample moment identity}
\label{sec:ix-t1}

\emph{Design.} Sections~\ref{sec:lsd}--\ref{sec:res} are asymptotic, whereas
Proposition~\ref{prop:cost} and Corollary~\ref{coro:frob-order} are not, and
are therefore tested separately at a single fixed configuration in which no
limit is taken. One fixed pair $(X,w)$ with $d=150$ and $N=300$, for which the
realized $R_N=124$, is held constant while only the resampling randomness is
redrawn, over $n_{\mathrm{rep}}=1000$ repetitions, for the multinomial and
residual schemes. The residual count covariance is
$\operatorname{Cov}(\bm{n})=\operatorname{diag}(u)-uu^\top/R_N$, with the
branch $R_N=0$ defined to be zero. The cost is evaluated as
\begin{equation}
  \operatorname{tr} S_r^2
    = \frac{1}{N^2}\,\bm{n}^\top (G\circ G)\,\bm{n},
  \label{eq:t1form}
\end{equation}
which requires no eigendecomposition and therefore tests the identity in the
form in which it is proved.

\emph{Results.} At this configuration the standardized deviations are
$1.948$ for multinomial resampling, with Monte Carlo mean $74.114\pm0.520$
against a theoretical cost of $75.128$, and $0.123$ for residual resampling,
with $31.299\pm0.248$ against $31.329$; both lie within the pre-registered
criterion $|\cdot|<3$, so the residual covariance formula is confirmed in the
exact form used in the proof. The covariance difference of
Corollary~\ref{coro:frob-order} is positive semidefinite to machine precision,
with minimum eigenvalue $-2.7\times10^{-15}$, and since $G\circ G\succeq0$ the
finite-sample Frobenius ordering between the two schemes follows. The statement
is conditional and distribution-free, and requires none of the asymptotic
assumptions.

\subsection{Robustness checks}
\label{sec:ix-appendix}

Three checks are summarized here and reported in full in the supplementary
material~(S.27). Repeating the E1 measurements at $N=500$ and $N=2000$ shows every
quantity converging toward its prediction with diminishing error bars, the
uniform-profile count zero mass moving from $0.3646\pm0.0051$ to
$0.3701\pm0.0020$ against $e^{-1}$. Replacing Gaussian entries by variance-one
Rademacher and Student-$t_8$ entries leaves the second spectral moments at
$2.821\pm0.017$ and $2.821\pm0.021$, against the finite-$N$ predictions $2.798$
and $2.801$ of Proposition~\ref{prop:popcost}, which depend on the entry
kurtosis $\kappa_4$; the deviations are $1.3\sigma$ and $0.9\sigma$. This is an
agreement report for a single configuration and is not offered as evidence of
universality, and no universality claim is made in this paper. Finally, for
$H=\tfrac12\delta_{1/2}+\tfrac12\delta_{3/2}$ the separable prediction
$\mu_S=H\boxtimes\mathcal{T}_c(\operatorname{Poi}(1))$ of
Theorem~\ref{thm:lsd-gen} is verified end to end, with post-resampling
$m_2=2.856$ against $2.85$, a zero atom of $0.2079$ against $0.209849$, and the
$\boxtimes$ pipeline matching a directly simulated Wishart spectrum to a
maximum Stieltjes deviation of $5.1\times10^{-4}$.


\section{Discussion and Limitations}
\label{sec:limits}

The results above concern four distinct objects, which should not be
conflated\cite{douc2005comparison}: the rank and zero-atom transition, described by $\mu(\{0\})$ with
threshold $p_{\mathsf{s}}$; the universal Marchenko--Pastur lower-edge
benchmark $|\sqrt c-\sqrt{p_{\mathsf{s}}}|^2$; the regular lower and internal
spectral edges, characterized by a branch-restricted critical equation; and
two notions of conditioning, namely the raw pseudo-condition number
$\kappa^+(S)$ and the bulk quantile condition number
\begin{equation}
  \kappa_{\alpha,\beta}^{\mathrm{bulk}}
    :=Q_{\mu^+}(\beta)/Q_{\mu^+}(\alpha),\qquad 0<\alpha<\beta<1,
  \label{eq:bulkcond}
\end{equation}
where $\mu^+:=\mu|_{(0,\infty)}/(1-\mu(\{0\}))$.

The scope of the two conditioning notions depends on the mechanism. For
multinomial resampling in the identity-population model the limiting right
support is unbounded by Proposition~\ref{prop:edges} and
$\kappa^+(S_{r,N})\to\infty$ in probability, since
$\lambda_{\max}(S_{r,N})\to\infty$ while
$\lambda_{\min}^+(S_{r,N})=O_{\mathbb P}(1)$. Fixed-quantile ratios therefore
describe bulk conditioning more robustly than the raw pseudo-condition number,
and a positive lower gap does not imply asymptotic well-conditioning. Residual
limits may instead have bounded support, as in the exact-uniform-weight case,
where $\kappa^+$ has a finite limit for $c\neq1$. The spectral cost considered
in this paper refers to these specific observables and not to generic
performance metrics: given $c$ and $H$, ESS determines the multinomial
limiting second spectral moment and bounds occupancy from below, but does not
determine the full limiting spectral distribution.

Five limitations delimit the scope of the results.

First, the analysis treats a single resampling step with exogenous weights.
Endogenous, likelihood-driven weights generally invalidate the independence
assumptions used in the covariance-spectrum step, so the spectral-map argument
does not extend automatically to that setting. Conditionally on the realized
weights the multinomial sampling law is unchanged, so suitable count-law
limits may survive under corresponding random-weight assumptions.

Second, the limit theorems cover multinomial resampling under (A4) and
residual resampling under (A4R). Establishing count-law limits for stratified
and systematic resampling requires additional analysis of ordering and
dependence and lies outside the scope of this paper; a random-ordering
assumption may prove useful, but is not claimed to be sufficient. These two
schemes accordingly do not appear in the numerical study.

Third, centering is not performed. It perturbs each matrix by a term of rank
at most one and therefore does not affect the limiting spectral distribution
as $d\to\infty$ whenever that distribution exists, but the finite-sample
conditional inequalities are sensitive to it.

Fourth, Theorem~\ref{thm:lsd-gen} is stated under Gaussianity alone, and no
universality claim is made; the entry-distribution check of
Section~\ref{sec:ix-appendix} covers a single configuration and is not a
universality result.

Fifth, the results establish no ordering of filtering accuracy, closed-loop
performance, or optimal resampling schedules: the experiments verify the
objects that the theory defines and no others. The critical-point readings in
Fig.~\ref{fig:e4spec} are finite-$N$ measurements taken near a transition, and
are reported together with their finite-size band rather than smoothed. The
spectral recursion of a complete filtering cycle remains for future work.

\section{Conclusion}
\label{sec:conclusion}

Resampling mechanisms admit a classification by their count-law limits.
Multinomial resampling is determined by the weight law $\rho$ alone, whereas
residual resampling retains information about the approach to the integer
boundaries through the joint law $\pi$, classified completely under
$\rho=\delta_1$ as the family
$(1-\alpha)\delta_1+\alpha\operatorname{Poi}(1)$. Vanishing weight
heterogeneity therefore does not imply vanishing residual randomization.

Composed with a classical covariance-side spectral map, these count laws yield
exact spectral consequences. For multinomial resampling in the
identity-population model they yield a rank-and-spectral-gap transition with
critical support touching, computable regular spectral edges, unbounded right
support with diverging pseudo-condition number, and an exact second-moment
cost. For residual resampling they yield the smaller cost
$c\,a_1^2\,\mathbb{E}_\pi U$, an asymptotic rank-preservation ordering, and a
complete one-parameter classification of the limits under vanishing
heterogeneity. Numerical experiments against pre-registered predictions
confirm each level of the theory.

These results support the use of spectral diagnostics alongside ESS in the
assessment of resampling for high-dimensional ensembles. Their incorporation
into adaptive resampling rules for likelihood-driven particle filters remains
an important direction for future work.

\clearpage
\onecolumn
\begin{center}
{\Large\bfseries Supplementary Material}\\[1.2ex]
{\large Proofs for `‘Spectral Costs of Multinomial and Residual Resampling
in High-Dimensional Particle Ensembles: A Free-Probability Analysis`’}\\[1.2ex]
{\normalsize SONG Deming}
\end{center}
\vspace{1.2em}

This supplement contains complete proofs of all results stated in the main text (Sections~S.0--S.22, organized by logical dependency; each section header names the main-text result it proves), followed by the solver unit tests and the full numerical tables referenced from the main text (Sections~S.23--S.27).

\setcounter{section}{-1}
\renewcommand{\thesection}{S.\arabic{section}}
\renewcommand{\theHsection}{S.\arabic{section}}


\section{Tool Inventory}\label{sec:S0}

\textbf{Tool S.1 (Rank inequality; \cite[Lemma~2.2]{bai2010spectral}).} For Hermitian $d\times d$ matrices $A,B$, $\|F^A-F^B\|_\infty\le\operatorname{rank}(A-B)/d$.

\textbf{Tool S.2 (Trace convexity chain; \cite{davis1957schwarz,fan1949theorem,karamata1932inegalite,marshall1979inequalities,bhatia1997matrix}).} Ky Fan majorization; Karamata; Davis: $\varphi$ convex $\Rightarrow$ $M\mapsto\operatorname{tr}\varphi(M)$ convex.

\textbf{Tool S.3 (Extreme eigenvalues of sample covariance matrices; \cite{bai1993limit}).} Let $Z\in\mathbb R^{d\times M}$ have i.i.d.\ entries drawn from a fixed distribution with mean $0$, variance $1$, finite fourth moment, and $d/M\to y\in(0,\infty)$. Then the extreme eigenvalues of $\frac1MZZ^\top$ converge in probability to the edges of $\mathrm{MP}(y)$: $\lambda_{\max}\to(1+\sqrt y)^2$; for $y<1$, $\lambda_{\min}\to(1-\sqrt y)^2$. For $y>1$, the Gram form gives directly: $\lambda_{\min}\big(\frac1MZ^\top Z\big)\to(\sqrt y-1)^2>0$ in probability—implying simultaneously that $Z$ has full column rank with probability tending to one and that the smallest positive eigenvalue of the sample covariance converges.

\section{Proof of Lemma~\ref{lem:unbiased}}\label{sec:S1}

$\mathbb E[S_{r,N}\mid X_N,w]=\frac1N\sum_i\mathbb E[n_i\mid X_N,w]\,x_ix_i^\top=\sum_iw_ix_ix_i^\top=S_{w,N}$, by linearity only. $\square$

\section{Proof of Theorem~\ref{thm:jensen}}\label{sec:S2}

The count vector of any scheme takes finitely many values $\boldsymbol n^{(1)},\dots,\boldsymbol n^{(L)}$ with probabilities $p_\ell$. By Lemma~\ref{lem:unbiased}, $\sum_\ell p_\ell S(\boldsymbol n^{(\ell)})=S_{w,N}$, and ordinary finite Jensen gives
\[
\mathbb E\big[F(S_{r,N})\,\big|\,X_N,w\big]=\sum_{\ell}p_\ell F\big(S(\boldsymbol n^{(\ell)})\big)\ \ge\ F(S_{w,N}).
\]
For strictly convex $F$, equality forces $S(\boldsymbol n^{(\ell)})$ identical across the support, i.e.\ $S_{r,N}=S_{w,N}$ a.s.

\textbf{Admissible functionals (paragraph following Theorem~\ref{thm:jensen} in the main text).} (a) Tool S.2 (the Davis trace-convexity step); for $q>2$, $x\mapsto x^q$ is not operator convex on $(0,\infty)$, whereas trace convexity still applies. (b) $\lambda_{\max}(M)=\sup_{\|v\|=1}v^\top Mv$, a supremum of linear functions. (c) $-\log(x+\varepsilon)$ is convex on $[0,\infty)$; the conclusion reads $\mathbb E\log\det(S_{r,N}+\varepsilon I)\le\log\det(S_{w,N}+\varepsilon I)$, with the convention $0\log0=0$ throughout. $\square$

\section{Proof of Proposition~\ref{prop:cost}, Proposition~\ref{prop:popcost}, and Corollary~\ref{coro:frob-order}}\label{sec:S3}

\textbf{Proposition~\ref{prop:cost}.} With $\delta_i=n_i/N-w_i$, $\Delta:=S_{r,N}-S_{w,N}=\sum_i\delta_ix_ix_i^\top$:
\[
\operatorname{tr}S_{r,N}^2-\operatorname{tr}S_{w,N}^2=2\operatorname{tr}(S_{w,N}\Delta)+\operatorname{tr}\Delta^2,\qquad \operatorname{tr}\Delta^2=\|\Delta\|_F^2.
\]
The linear term vanishes in conditional mean (Lemma~\ref{lem:unbiased}), so
\[
\mathbb E\big[\operatorname{tr}S_{r,N}^2-\operatorname{tr}S_{w,N}^2\,\big|\,X_N,w\big]=\mathbb E\big[\|\Delta\|_F^2\,\big|\,X_N,w\big]=\sum_{i,j}\mathbb E[\delta_i\delta_j]G_{ij}^2=\frac1{N^2}\operatorname{tr}\big[(G\circ G)\,\operatorname{Cov}(\boldsymbol n\mid X_N,w)\big],
\]
nonnegative since $G\circ G$ and the covariance are PSD (Schur product theorem).

\textbf{Multinomial specialization.} With $\mathbb E[n_i^2]=Nw_i(1-w_i)+N^2w_i^2$ and $\mathbb E[n_in_j]=N(N-1)w_iw_j$ ($i\neq j$), substitution gives $\frac1N[\sum_iw_i\|x_i\|^4-\operatorname{tr}S_{w,N}^2]$. For the equivalent geometric form, set $Y_i:=x_ix_i^\top$: since $\sum_iw_i=1$ and $\sum_iw_iY_i=S_{w,N}$,
\[
\sum_iw_i\big\|Y_i-S_{w,N}\big\|_F^2=\sum_iw_i\|Y_i\|_F^2-\Big\|\sum_iw_iY_i\Big\|_F^2,
\]
so the cost equals $\frac1N\sum_iw_i\|x_ix_i^\top-S_{w,N}\|_F^2$, which vanishes iff $x_ix_i^\top=S_{w,N}$ for every $i$ with $w_i>0$ (covering all-zero particles and the $\pm v$ collinear case; in the latter, $S_{r,N}=\frac1N(\sum_in_i)vv^\top=S_{w,N}$ a.s.\ even with random counts). $\square$

\textbf{Proposition~\ref{prop:popcost}.} For $x_i=\Sigma^{1/2}z_i$ with kurtosis $\kappa_4$:
\[
\mathbb E\|x_i\|^4=(\operatorname{tr}\Sigma)^2+2\operatorname{tr}\Sigma^2+\kappa_4\sum_j\Sigma_{jj}^2,\qquad \mathbb EG_{ij}^2=\operatorname{tr}\Sigma^2\ \ (i\neq j).
\]
Substituting into the multinomial formula and taking the expectation over the particle matrix (weights deterministic or independent):
\[
\mathbb E\big[\operatorname{tr}S_{r,N}^2-\operatorname{tr}S_{w,N}^2\,\big|\,w\big]=\frac{1-s_{2,N}}{N}\Big[(\operatorname{tr}\Sigma_N)^2+\operatorname{tr}\Sigma_N^2+\kappa_4\sum_j(\Sigma_N)_{jj}^2\Big],
\]
exactly. For the asymptotics, $\frac1d\operatorname{tr}\Sigma_N^2\le\|\Sigma_N\|_{\rm op}\frac{\operatorname{tr}\Sigma_N}{d}=O(1)$ (equivalently $\le\|\Sigma_N\|_{\rm op}^2$) and $\frac1d\sum_j(\Sigma_{N,jj})^2\le\|\Sigma_N\|_{\rm op}^2=O(1)$, both under (A5) at finite $N$ (no limit quantities); hence
\[
\frac1d\mathbb E[\mathrm{cost}\mid w]=c_N\Big(\frac{\operatorname{tr}\Sigma_N}{d}\Big)^2(1-s_{2,N})+O(N^{-1})=c\,a_1^2(1-s_{2,N})+o(1),
\]
where no convergence rate is assumed for $c_N$ or the population first moment, so the total error is only $o(1)$. $\square$

\textbf{Corollary~\ref{coro:frob-order} (finite-sample Frobenius ordering).}
Write $\lambda=k+u$ with $L=\sum_ik_i$, $R=\sum_iu_i$, $L,R>0$, and note
$L+R=N$. In these coordinates
$\operatorname{Cov}_{\rm multi}=\operatorname{diag}(\lambda)-\frac1N\lambda\lambda^\top$
and
$\operatorname{Cov}_{\rm res}=\operatorname{diag}(u)-\frac1R uu^\top$,
so, using $\operatorname{diag}(\lambda)-\operatorname{diag}(u)=\operatorname{diag}(k)$,
\[
\operatorname{Cov}_{\rm multi}-\operatorname{Cov}_{\rm res}
=\operatorname{diag}(k)-\tfrac1N\lambda\lambda^\top+\tfrac1R uu^\top.
\]
Expanding the claimed rank-one term with $LR/N$ and using $N=L+R$,
\[
\tfrac{LR}{N}\big(\tfrac kL-\tfrac uR\big)\big(\tfrac kL-\tfrac uR\big)^\top
=\tfrac{R}{LN}kk^\top-\tfrac1N\big(ku^\top+uk^\top\big)+\tfrac{L}{RN}uu^\top,
\]
and the $kk^\top$, $ku^\top+uk^\top$, $uu^\top$ coefficients of
$\operatorname{diag}(k)-\frac1N\lambda\lambda^\top+\frac1Ruu^\top$
minus $\operatorname{diag}(k)$ are respectively $-\frac1N$, $-\frac1N$,
$-\frac1N+\frac1R$, while those of
$-\frac1L kk^\top+\frac{LR}{N}(\cdots)$ are
$-\frac1L+\frac{R}{LN}=-\frac1N$, $-\frac1N$, $\frac{L}{RN}=\frac1R-\frac1N$;
the two sides agree, giving the stated identity. The first term
$\operatorname{diag}(k)-kk^\top/L$ is the covariance of a
$\mathrm{Multinomial}(1;k/L)$ draw, hence positive semidefinite, and the
second term is an explicit positive semidefinite rank-one matrix, so the
difference is positive semidefinite. The boundary cases are immediate: if
$R=0$ then $\operatorname{Cov}_{\rm res}=0$; if $L=0$ then $N=R$ and the two
covariances coincide. The ordering of the conditional expected error energies
follows from the quadratic form of Proposition~\ref{prop:cost}, since
$G\circ G\succeq0$ (Schur product theorem) and
$A\succeq B\succeq0$ implies $\operatorname{tr}[MA]\ge\operatorname{tr}[MB]$
for $M\succeq0$. $\square$

\section{Proof of Theorem~\ref{thm:multi-count} (Multinomial count law)}\label{sec:S4}

\textbf{Bounded test functions.} Represent the counts by independent ancestor indices $J_1,\dots,J_N$ with $\mathbb P(J_a=i)=w_i$. Step 1 (expectation): the total-variation bound for the Poisson approximation is
\[
d_{\rm TV}\big(\operatorname{Bin}(N,\lambda/N),\ \operatorname{Poi}(\lambda)\big)\ \le\ \frac{\lambda^2}{N},
\]
which is $O(M^2/N)$ uniformly on $\lambda\in[0,M]$; the tail above $M$ is controlled by (A4) via $\rho_N((M,\infty))\le M^{-2}\int\lambda^2\rho_N\to M^{-2}r_2$. Hence $\frac1N\sum_i\mathbb E f(n_i)\to\int\mathbb E f(\operatorname{Poi}(\lambda))\rho(d\lambda)$. Step 2 (concentration): replacing one ancestor label changes two counts, so $\frac1N\sum_if(n_i)$ changes by at most $4\|f\|_\infty/N$ (equivalently $2\operatorname{osc}(f)/N$); bounded differences give variance $O(N^{-1})$. (If weights are random, the computations above are conditional on $w$.)

\textbf{Second moment (factorial-moment form).} Let $B_N:=\frac1N\sum_in_i(n_i-1)$, so $\frac1N\sum_in_i^2=1+B_N$. Conditionally on $w$,
\[
\mathbb E[B_N\mid w]=(N-1)\sum_iw_i^2=\Big(1-\frac1N\Big)\int\lambda^2\rho_N(d\lambda),
\]
and expanding the collision indicators gives
\[
\operatorname{Var}(B_N\mid w)\ \le\ 2\sum_iw_i^2+4N\sum_iw_i^3\ \longrightarrow\ 0,
\]
since $N\sum_iw_i^3\le(\max_iw_i)\,N\sum_iw_i^2\to0$ under (A4) (no third-moment assumption on the limiting weight law is needed). Hence $\frac1N\sum_in_i^2\to1+r_2=b_2$ in probability. $\square$

\section{The Poissonization Coupling Lemma}\label{sec:S5}

This lemma is stated and used in the main text in the unnumbered paragraphs following Theorem~\ref{thm:multi-count} and Theorem~\ref{thm:lsd-id}; it is stated formally and proved here.

\textbf{Distributional identity.} For independent $n_i'\sim\operatorname{Poi}(\mu_i)$ with $R:=\sum_i\mu_i$ a positive integer and $T=\sum_in_i'$, for any $k$ with $\sum k_i=R$,
\[
\mathbb P(\boldsymbol n'=k\mid T=R)=\frac{\prod_ie^{-\mu_i}\mu_i^{k_i}/k_i!}{e^{-R}R^R/R!}=\frac{R!}{\prod_ik_i!}\prod_i\Big(\frac{\mu_i}{R}\Big)^{k_i}=\mathrm{Multinomial}(R;\mu/R)\{k\}.
\]
\textbf{Coupling (label form).} Let $T\sim\operatorname{Poi}(R)$ independently of an infinite i.i.d.\ label sequence $J_1,J_2,\dots$ with $\mathbb P(J_a=i)=\mu_i/R$. Define $n_i'$ as the number of occurrences of $i$ among the first $T$ labels, and $n_i$ among the first $R$. Then $(n_i')$ are independent $\operatorname{Poi}(\mu_i)$, $(n_i)\sim\mathrm{Multinomial}(R;\mu/R)$, and $\sum_i|n_i-n_i'|=|T-R|$ exactly (the two count vectors differ only in the last $|T-R|$ labels). The matrix difference is a signed sum of $|T-R|$ rank-one terms, so $\operatorname{rank}(S_N-S_N')\le|T-R|$; Tool S.1 gives $\|F^{S_N}-F^{S_N'}\|_\infty\le|T-R|/d$, which tends to zero in probability provided $\sqrt R/d\to0$ (satisfied here since $R\le N$ and $d/N\to c>0$). This coupling controls weak spectral convergence; it does not by itself control second spectral moments. $\square$

\section{Proof of Theorem~\ref{thm:lsd-id} (LSD, identity population)}\label{sec:S6}

\textbf{Step 1 (Poissonization).} The coupling lemma of Section~\ref{sec:S5} + Tool S.1: multinomial and independent-Poisson models share the same LSD limit; assume $D_N=\operatorname{diag}(n_i)$ with independent $\operatorname{Poi}(\lambda_i)$ entries.

\textbf{Step 2 (truncation, in-probability two-limit scheme).} Fix $M$ and set $D_N^{(M)}:=\operatorname{diag}(n_i\wedge M)$. Then
\[
\operatorname{rank}\Big(\frac1N Z(D_N-D_N^{(M)})Z^\top\Big)\le\#\{i:n_i>M\},\qquad \frac1N\#\{i:n_i>M\}\xrightarrow{\mathbb P}\nu((M,\infty)),
\]
so for every $\varepsilon>0$,
\[
\mathbb P\Big(\big\|F^{S_{r,N}}-F^{S_{r,N}^{(M)}}\big\|_\infty>\frac{\nu((M,\infty))}{c}+\varepsilon\Big)\to0.
\]
For each fixed $M$, Step 3 identifies the truncated model's limit as $\mathcal T_c(\nu^{(M)})$; the two-limit argument is then performed in Lévy or bounded-Lipschitz distance (not KS distance: the limit may have atoms), letting $M\to\infty$ with $\nu((M,\infty))\to0$ and continuity of the fixed-point law in the population law.

\textbf{Step 3 (Silverstein convergence; random-diagonal transfer).} With $W_N^{(M)}:=\frac1d(D_N^{(M)})^{1/2}Z^\top Z(D_N^{(M)})^{1/2}\in\mathbb R^{N\times N}$: the empirical law of $D_N^{(M)}$ converges in probability, not per realization. From any subsequence, extract a further subsequence along which this convergence is almost sure; conditional on such a realization, apply the deterministic-population covariance theorem \cite{silverstein1995strong,bai2010spectral}; bounded convergence of the conditional error probabilities then yields convergence in probability along the original subsequence. Hence $\mu_{W_N^{(M)}}\Rightarrow\mu_W^{(M)}=\nu^{(M)}\boxtimes\mathrm{MP}(1/c)$ in probability.

\textbf{Step 4 (exact companion bookkeeping).} Let $Q_N^{(M)}:=\frac1NZD_N^{(M)}Z^\top\in\mathbb R^{d\times d}$ and $B_N^{(M)}:=\frac1N(D_N^{(M)})^{1/2}Z^\top Z(D_N^{(M)})^{1/2}=\frac dNW_N^{(M)}\in\mathbb R^{N\times N}$. With $A=Z(D_N^{(M)})^{1/2}$, $B=(D_N^{(M)})^{1/2}Z^\top$, the products $AB=NQ_N^{(M)}$ and $BA=NB_N^{(M)}$ share nonzero spectra, and with $r=\operatorname{rank}$, $c_N:=d/N$:
\[
\mu_{B_N^{(M)}}=c_N\,\mu_{Q_N^{(M)}}+\big(1-c_N\big)\delta_0\qquad\text{(exact, every }N\text{)};
\]
for $c_N>1$ the negative coefficient is cancelled by the forced zero atom of $\mu_{Q_N^{(M)}}$.

\textbf{Step 5 (assembly).} $\mu_{B_N^{(M)}}\Rightarrow D_c(\mu_W^{(M)})$, so the truncated model converges to $\mathcal T_c(\nu^{(M)})$; Step 2's two-limit scheme yields $\mu_Q=\mathcal T_c(\nu)$. $\square$

\textbf{Equivalence of the two forms of Definition~\ref{def:Tc}.} The companion transform of the $W$-model satisfies $\underline m_W(z)=c\,m_Q(cz)$ (from Step 4's scaling) and, by the classical companion form of Silverstein's equation with ratio $y=1/c$,
\[
z=-\frac1{\underline m_W(z)}+\frac1c\int\frac{t}{1+t\,\underline m_W(z)}\,\nu(dt).
\]
Substituting $\underline m_W(z)=c\,m_Q(cz)$, multiplying by $c$, and setting $\zeta=cz$ gives
\[
\zeta=-\frac1{m_Q(\zeta)}+\int\frac{t}{1+ct\,m_Q(\zeta)}\,\nu(dt),
\]
the covariance-side fixed-point equation. $\square$

\textbf{Normalization check.} $\eta=\delta_1$: $B_N=\frac1NZ^\top Z$ has LSD $c\,\mathrm{MP}(c)+(1-c)\delta_0$ by Step 4's counting applied to the standard Wishart; independently, the covariance-side equation with $\nu=\delta_1$ reduces to the MP fixed point $czm^2+(z+c-1)m+1=0$. Equating gives $\mathcal T_c(\delta_1)=\mathrm{MP}(c)$ and $D_c(\mathrm{MP}(1/c))=c\,\mathrm{MP}(c)+(1-c)\delta_0$. (Moment recursion: $m_2=1+cb_2$, and $m_3=1+3cb_2+c^2b_3$ whenever the corresponding moments are finite.) $\square$

\section{Proof of Theorem~\ref{thm:lsd-gen} (General population spectrum; Gaussian case)}\label{sec:S7}

Order of argument: (i) truncate counts to $\nu^{(M)}$ as in S.\ref{sec:S6}; the truncated $Q_N^{(M)}$ satisfies $\|Q_N^{(M)}\|_{\rm op}\le M\|\frac1NZZ^\top\|_{\rm op}=O_{\mathbb P}(1)$ (not a deterministic uniform bound over Gaussian realizations); (ii) given $D_N^{(M)}$ (independent of $Z_N$ by exogeneity), Gaussianity makes $Q_N^{(M)}$ conditionally orthogonally invariant, hence asymptotically free from the deterministic $\Sigma_N$ \cite{collins2006integration}; the probability modes are aligned as follows: the count-law convergence is in probability (S.\ref{sec:S6} Step 3's subsequence bridge), the freeness conclusion is applied conditionally, and the final LSD convergence is in probability; (iii) the truncated model's LSD is $H\boxtimes\mathcal T_c(\nu^{(M)})$; (iv) remove the truncation by Tool S.1 and pass through weak continuity of the free multiplicative convolution to $H\boxtimes\mathcal T_c(\nu)$. The matrices $\Sigma_N^{1/2}Q_N\Sigma_N^{1/2}$ and $Q_N^{1/2}\Sigma_NQ_N^{1/2}$ share their nonzero spectrum; for the bounded truncated model this relation, rather than any moment-sharing statement, is what the proof uses. $\square$

\section{Proof of Proposition~\ref{prop:atom} (Zero-eigenvalue mass)}\label{sec:S8}

\textbf{Rank loss (finite sample).} $S_{r,N}=\frac1N\sum_{i:n_i>0}n_iz_iz_i^\top$ has at most $\#\{i:n_i>0\}$ nonzero rank-one summands, so $\operatorname{rank}(S_{r,N})\le\min(d,\#\{i:n_i>0\})$; equality a.s.\ for continuously distributed $Z$ (general position). Note that a finite-sample rank proportion alone does not exclude a non-vanishing fraction of nonzero eigenvalues accumulating at zero (a vanishing fraction cannot form an atom, but a non-vanishing fraction tending to zero would matter), which is why the LSD statement is proved through the transform.

\textbf{LSD zero mass (companion-law max-atom rule).} With $q:=\nu(\{0\})$, the max-atom rule for free multiplicative convolution \cite{belinschi-atoms,bercovici1993free} applied to the companion law gives
\[
\mu_B(\{0\})=\max\big\{q,\ (1-c)_+\big\},\qquad \mu_B=D_c\big(\nu\boxtimes\mathrm{MP}(1/c)\big),
\]
since $\mathrm{MP}(1/c)$ has zero mass $(1-c)_+$ (present iff $c<1$). Reading off from $\mu_B=c\mu_Q+(1-c)\delta_0$:
\[
\mu_Q(\{0\})=\frac{\max\{q,(1-c)_+\}-(1-c)}{c}=\max\Big(0,\ 1-\frac{1-q}{c}\Big).
\]
For the separable case, $\mu_S(\{0\})=\max\{H(\{0\}),\mu_Q(\{0\})\}$ is the same rule applied to $H\boxtimes\mu_Q$; the alternative rank route would require $\operatorname{rank}(\Sigma_N)/d\to1-H(\{0\})$, which fails under (A5) alone (e.g.\ $\Sigma_N=N^{-1}I_d$), and is therefore not used. $\square$

\section{Proof of Proposition~\ref{prop:moments} (Spectral moments; convergence in probability)}\label{sec:S9}

\textbf{Expectation identity.} Let $Y:=\Sigma^{1/2}ZDZ^\top\Sigma^{1/2}$ and $S=Y/N$; write everything as $\mathbb E\operatorname{tr}Y^2/(dN^2)$. Conditional Gaussian Wick evaluation given $D$:
\[
\mathbb E\big[\operatorname{tr}Y^2\,\big|\,D\big]=\big((\operatorname{tr}\Sigma)^2+\operatorname{tr}\Sigma^2\big)\sum_id_i^2+\operatorname{tr}\Sigma^2\Big(\sum_id_i\Big)^2.
\]
For non-Gaussian entries an extra term $\kappa_4\sum_id_i^2\sum_j\Sigma_{jj}^2$ appears, which after normalization by $dN^2$ is $O(N^{-1})$ under (A5) together with $\frac1N\sum_id_i^2=O(1)$ (in probability or in expectation, as recorded with each result).

\textbf{Convergence in probability (auxiliary lemma).} \emph{Under the standing assumptions, $\frac1d\operatorname{tr}S_N\xrightarrow{\mathbb P}a_1$ and $\frac1d\operatorname{tr}S_N^2\xrightarrow{\mathbb P}m_{2,H}+c\,a_1^2\cdot(\text{diag second moment})$.}

\emph{First moment.} Given $D$, $\mathbb E_Z\big[\frac1d\operatorname{tr}S_N\,\big|\,D\big]=\frac{\operatorname{tr}\Sigma_N}{d}\cdot\frac{\sum_id_i}{N}$ and
\[
\operatorname{Var}_Z\Big(\frac1d\operatorname{tr}S_N\,\Big|\,D\Big)=\frac{2\operatorname{tr}\Sigma_N^2}{d^2N^2}\sum_id_i^2\ \longrightarrow\ 0,
\]
since $\sum_id_i=N$, $\frac1N\sum_id_i^2=O(1)$, and $\frac1d\operatorname{tr}\Sigma_N^2=O(1)$. Conditional Chebyshev plus convergence of the conditional mean give the claim for both pre and post matrices.

\emph{Second moment, two ingredients.} (i) \textbf{Truncated model.} With $d_i^{(M)}:=d_i\wedge M$,
\[
\big\|S_N^{(M)}\big\|_{\rm op}\le M\,\|\Sigma_N\|_{\rm op}\,\Big\|\frac1NZZ^\top\Big\|_{\rm op}=O_{\mathbb P}(1),
\]
so the truncated model has asymptotically bounded spectrum, and its LSD convergence (to $H\boxtimes\mathcal T_c(\nu^{(M)})$, by S.\ref{sec:S7}) upgrades to convergence in probability of its second spectral moment, whose limit is
\[
m_{2,H}\Big(\int t\,\nu^{(M)}(dt)\Big)^2+c\,a_1^2\int t^2\,\nu^{(M)}(dt)
\]
(the truncated law has mean $\int t\,\nu^{(M)}(dt)$, generally less than one; the mean-one formula does not apply).

(ii) \textbf{Truncation-error control.} Set $X_{N,M}:=\frac1d\big[\operatorname{tr}S(D)^2-\operatorname{tr}S(D^{(M)})^2\big]$. Since $D\succeq D^{(M)}\succeq0$, one has $S(D)\succeq S(D^{(M)})\succeq0$, and
\[
\operatorname{tr}A^2-\operatorname{tr}B^2=2\operatorname{tr}\big[B(A-B)\big]+\operatorname{tr}(A-B)^2\ge0\qquad(A\succeq B\succeq0),
\]
so $X_{N,M}\ge0$. With $\bar d:=\frac1N\sum_id_i$ and $\bar d_M:=\frac1N\sum_id_i^{(M)}$, the Wick identity applied to both matrices gives
\[
\mathbb E_Z\big[X_{N,M}\,\big|\,D\big]=\frac{(\operatorname{tr}\Sigma)^2+\operatorname{tr}\Sigma^2}{dN}\cdot\frac1N\sum_i\big(d_i^2-(d_i^{(M)})^2\big)+\frac{\operatorname{tr}\Sigma^2}{d}\big(\bar d^2-\bar d_M^2\big),
\]
and both terms vanish as $M\to\infty$ by the second tail moment of the count law (finite second moment; the mean difference $\bar d-\bar d_M\to0$ likewise). To remove the truncation without assuming unconditional integrability of the random tail, let $A_{N,M}(D):=\mathbb E_Z[X_{N,M}\mid D]$; since $X_{N,M}\ge0$, for any $\varepsilon,\delta>0$,
\[
\mathbb P(X_{N,M}>\varepsilon)\ \le\ \mathbb P(A_{N,M}>\delta)+\frac{\delta}{\varepsilon},
\]
because $\mathbb P(X>\varepsilon)\le\mathbb P(A>\delta)+\mathbb P(X>\varepsilon,\,A\le\delta)$ and the second term is at most $\mathbb E[A\mathbf 1_{A\le\delta}]/\varepsilon\le\delta/\varepsilon$. First let $N\to\infty$ at fixed $M$, then $M\to\infty$, then $\delta\downarrow0$; the truncation is removed.

The limit value follows from the expectation identity: pre-resampling ($D=\Lambda_N$): diagonal second moment $\to r_2$; post-resampling ($D=D_N$): $\to b_2=1+r_2$ (Theorem~\ref{thm:multi-count}). Increment $c\,a_1^2$. $\square$

\section{Proof of Proposition~\ref{prop:ess} (ESS; occupancy bound; the pair)}\label{sec:S10}

\textbf{ESS.} $N/\mathrm{ESS}=N\sum_iw_i^2\to r_2$, equivalently $\mathrm{ESS}/N\to1/r_2$; given $c$ and $H$, it determines the multinomial limiting second spectral moment (Proposition~\ref{prop:moments}), but not the full limiting spectral distribution.

\textbf{Occupancy bound (one line).} For $K\sim\MixPoi(\rho)$, $\mathbb EK=1$ and $\mathbb EK^2=1+r_2$; by Cauchy--Schwarz,
\[
1=(\mathbb EK)^2=\big(\mathbb E[K\mathbf 1_{\{K>0\}}]\big)^2\le\mathbb EK^2\cdot\mathbb P(K>0),
\]
so $p_{\rm multi}\ge1/(1+r_2)$. If $r_2=1$ with mean $1$, then $\mathrm{Var}(K)=0$, forcing $\rho=\delta_1$ and $\nu=\operatorname{Poi}(1)$. $\square$

\textbf{The pair.} $\rho^{(A)}=\frac12\delta_{1/4}+\frac12\delta_{7/4}$: mean $1$; $\mathbb EL^2=25/16$; count-law zero mass $\frac12e^{-1/4}+\frac12e^{-7/4}=0.476287$; count third moment $b_3=2.6875+3(25/16)+1=8.375$. $\rho^{(B)}=\frac9{10}\delta_{3/4}+\frac1{10}\delta_{13/4}$: mean $1$; $\mathbb EL^2=25/16$; zero mass $\frac9{10}e^{-3/4}+\frac1{10}e^{-13/4}=0.429007$; $b_3=3.8125+3(25/16)+1=9.5$. Identical ESS $=16N/25$ exactly at multiples of $10$. Since $m_3(\mathcal T_c(\nu))=1+3cb_2+c^2b_3$ (moments finite here) with equal $b_2=41/16$ and different $b_3$,
\[
m_3^{(B)}-m_3^{(A)}=c^2\,(9.5-8.375)=\frac98c^2>0,
\]
so the two limiting spectra differ for every $c>0$—this example is self-contained and does not rely on the general injectivity lemma (S.\ref{sec:S21}). At $c=1$, $\mu_Q(\{0\})=\nu(\{0\})$ witnesses $0.476287$ vs.\ $0.429007$. (Differences in third spectral moments establish different distributions; they do not by themselves establish different support edges.) $\square$

\section{Proof of Corollary~\ref{coro:deff} (Effective dimension; IPR counterexample)}\label{sec:S11}

Under the assumptions of Proposition~\ref{prop:moments}, with $a_1>0$: S.\ref{sec:S9}'s convergence in probability of the normalized first and second moments, plus continuity of $\deff/d$ at the positive limit point, gives the two displayed limits and the explicit gap. No finite-sample ordering is claimed, because of the random denominator: with $d=N=2$, $x_1=3e_1$, $x_2=e_2$, $w=(0.9,0.1)$, $\operatorname{IPR}(S_w)=(8.1^2+0.1^2)/8.2^2\approx0.97591$, while multinomial resampling gives counts $(2,0),(1,1),(0,2)$ with probabilities $0.81,0.18,0.01$ and IPR values $1,0.82,1$, hence the \textbf{conditional expectation} $\mathbb E[\operatorname{IPR}(S_{r,N})\mid X,w]=0.82+0.18\times0.82=0.9676$—strictly smaller. $\square$

\section{Proof of Theorem~\ref{thm:rankgap} (Rank-and-spectral-gap transition; critical touching)}\label{sec:S12}

\textbf{Active-column reduction and domination.} With $I_N:=\{i:n_i>0\}$,
\[
S_{r,N}=\frac1N\sum_{i\in I_N}n_i\,z_iz_i^\top=\frac1N Z_{I_N}D_{\rm act}Z_{I_N}^\top,\qquad D_{\rm act}\succeq I_{M_N},
\]
since active counts are at least $1$; hence $S_{r,N}\succeq\widetilde W_N:=\frac1NZ_{I_N}Z_{I_N}^\top$ and Weyl monotonicity applies. Conditional on the active index set, $Z_{I_N}$ has i.i.d.\ entries with the original entry distribution, by independence of the counts and the particle matrix (exogeneity). $M_N/N\to p$ in probability by Theorem~\ref{thm:multi-count}. Random $M_N$ is handled by conditioning and a subsequence argument (as in S.\ref{sec:S6} Step 3).

\textbf{Subcritical $c<p$.} $\widetilde W_N=\frac{M_N}{N}\cdot\frac1{M_N}Z_{I_N}Z_{I_N}^\top$; by Tool S.3, $\lambda_{\min}\big(\frac1{M_N}Z_{I_N}Z_{I_N}^\top\big)\to(1-\sqrt{c/p})^2$ in probability, hence $\lambda_{\min}(\widetilde W_N)\to\frac{M_N}{N}(1-\sqrt{c/p})^2\to(\sqrt p-\sqrt c)^2$ in probability, and Weyl monotonicity yields, for every $\varepsilon>0$, $\mathbb P\big(\lambda_{\min}(S_{r,N})\ge(\sqrt p-\sqrt c)^2-\varepsilon\big)\to1$. Since the limit is positive, $\mathbb P(\operatorname{rank}(S_{r,N})=d)\to1$ follows.

\textbf{Supercritical $c>p$.} Pass to the Gram matrix $\Gamma_N:=\frac1NZ_{I_N}^\top Z_{I_N}=\frac dN\cdot\frac1dZ_{I_N}^\top Z_{I_N}$ ($M_N\times M_N$). The matrix $\frac1dZ_{I_N}^\top Z_{I_N}$ has population dimension $M_N$, sample size $d$, and ratio $M_N/d\to p/c<1$ (note: $c>p$), so by Tool S.3 with $y=p/c<1$, $\lambda_{\min}\big(\frac1dZ_{I_N}^\top Z_{I_N}\big)\to(1-\sqrt{p/c})^2$ in probability; scaling by $d/N\to c$ gives $\lambda_{\min}(\Gamma_N)\to c(1-\sqrt{p/c})^2=(\sqrt c-\sqrt p)^2$ in probability. The Gram lower-edge convergence also implies $\mathbb P(\operatorname{rank}(Z_{I_N})=M_N)\to1$ (a positive limiting smallest eigenvalue forces full column rank with high probability; almost-sure rank equalities require continuous entries). On the event $\operatorname{rank}(Z_{I_N})=M_N$, both $\widetilde W_N$ and $S_{r,N}$ have rank exactly $M_N$ (the latter by domination: $\operatorname{range}(\widetilde W_N)\subseteq\operatorname{range}(S_{r,N})$, and the active expansion gives the reverse inequality), and with the convention $\lambda_1\ge\cdots\ge\lambda_d$ (descending), the smallest positive eigenvalue is $\lambda_{M_N}$; Weyl monotonicity at that index gives
\[
\lambda_{\min}^{+}(S_{r,N})=\lambda_{M_N}(S_{r,N})\ \ge\ \lambda_{M_N}(\widetilde W_N)=\lambda_{\min}(\Gamma_N),
\]
hence $\mathbb P\big(\lambda_{\min}^+(S_{r,N})\ge(\sqrt c-\sqrt p)^2-\varepsilon\big)\to1$. The zero-atom mass is Proposition~\ref{prop:atom}.

\textbf{Criticality $c=p$.} The fixed-point identity of Definition~\ref{def:Tc}, defined initially on $\mathbb C^+$, extends to the negative real axis by continuity (no singularity for nonnegative spectral measures). Write $p=\nu((0,\infty))$, set $c=p$, take $z=-s<0$, and write $m=m(-s)>0$. The equation rearranges to
\[
sm(-s)=\frac1c\int_{(0,\infty)}\frac{\nu(dt)}{1+ctm(-s)}.
\]
\emph{Step 1 ($m(-s)\to\infty$).} Suppose along some sequence $s_j\downarrow0$ one had $m(-s_j)\le C$. Then $\frac{1}{1+ctm(-s_j)}\ge\frac{1}{1+ctC}$ pointwise, so the right side is bounded below by $\frac1c\int_{(0,\infty)}\frac{\nu(dt)}{1+ctC}>0$, while the left side $s_jm(-s_j)\to0$—contradiction. \emph{Step 2 (no atom).} Now that $m(-s)\to\infty$, the integrand tends to $0$ pointwise and is bounded by $1$; dominated convergence gives $sm(-s)\to0$, and since $sm(-s)=\int\frac{s}{x+s}\mu_Q(dx)\to\mu_Q(\{0\})$, this proves $\mu_Q(\{0\})=0$. \emph{Step 3 (touching).} If the nonzero spectrum were separated from zero by a gap $\varepsilon>0$, all mass would lie in $[\varepsilon,\infty)$ and $m(-s)=\int\frac{\mu_Q(dx)}{x+s}\le1/\varepsilon$, contradicting Step 1. Hence $\inf\operatorname{supp}\mu_Q=0$.

\textbf{Finite-sample statement at criticality.} For every $\varepsilon>0$, $\mu_Q((0,\varepsilon))>0$ by Step 3; choose $f\in C_c((0,\varepsilon))$, $f\ge0$, $\int f\,d\mu_Q>0$. Empirical spectral convergence gives $\int f\,d\mu_N\xrightarrow{\mathbb P}\int f\,d\mu_Q>0$, so with probability tending to one at least one eigenvalue lies in $(0,\varepsilon)$: $\mathbb P(\lambda_{\min}^+(S_{r,N})<\varepsilon)\to1$; nonnegativity gives the other direction, hence $\lambda_{\min}^+(S_{r,N})\xrightarrow{\mathbb P}0$. (No no-eigenvalues-outside-the-support input is needed for this statement; such inputs serve stronger claims, e.g.\ convergence to the true positive lower edge off criticality, or edge rates.)

\textbf{Divergence rate of the Stieltjes transform.} Multiplying the critical identity by $m(-s)$:
\[
s\,m(-s)^2=\frac1c\int_{(0,\infty)}\frac{m(-s)}{1+ctm(-s)}\,\nu(dt).
\]
As $s\downarrow0$, $\frac{m(-s)}{1+ctm(-s)}\to\frac1{ct}$ pointwise and $\frac{m(-s)}{1+ctm(-s)}\le\frac1{ct}$; for integer-valued counts, $t\ge1$ on the support, so $t^{-1}$ is integrable, and dominated convergence gives $s\,m(-s)^2\to\frac1{c^2}\int_{(0,\infty)}\frac1t\,\nu(dt)$, i.e.\ $m(-s)\sim\frac1c\big(\int t^{-1}\nu(dt)\big)^{1/2}s^{-1/2}$. (This is an asymptotic for the Stieltjes transform along the negative real axis; a density-level $x^{-1/2}$ asymptotic would require additional argument.)

\emph{Benchmark scope.} The domination argument discards the additional weights $D_{\rm act}-I$, so it need not identify the exact lower edge. The exact regular lower edge depends on the full count law and is characterized by Proposition~\ref{prop:edges}. $\square$

\section{Proof of Proposition~\ref{prop:edges} (Edges; unbounded support via positive fractions)}\label{sec:S13}

\textbf{Edge parametrization.} By \cite{silverstein1995analysis} (see also \cite[Ch.~6]{bai2010spectral}), the law $\mu_W$ characterized by the Silverstein equation admits the real parametrization $z(u)=-\frac1u+y\int\frac{t}{1+tu}\nu(dt)$, $y=1/c$, on real $u\neq0$ with $-1/u\notin\operatorname{supp}\nu$; the support-complement branches are those with $z'(u)>0$, and regular endpoints arise at boundary critical points of these branches. Since $S_{r,N}$ and $B_N=\frac dNW_N$ share nonzero spectra and $B_N$ is $W_N$ scaled by $d/N\to c$, the regular edges of the nonzero part of $\mu_Q$ are $c\,z(u^*)$. Consistency: $\nu=\delta_1$ gives $c(1\mp\sqrt{1/c})^2=(1\mp\sqrt c)^2$. The critical zero endpoint corresponds to $u\to\infty$ rather than a finite critical point and is treated separately by Theorem~\ref{thm:rankgap}(iii). $\square$

\textbf{Unbounded support.} Since $\rho$ has mean $1$, $\rho((0,\infty))>0$, so $\nu=\MixPoi(\rho)$ satisfies $\nu(\{k\})>0$ for every $k\in\mathbb Z_{\ge0}$; write $p_K:=\nu([K,\infty))>0$, noting $p_K\to0$. Fix a level $L>0$ and choose $K$ with $p_K<c/4$ and $Kc/4>L$. Set $A_K:=\{i:n_i\ge K\}$ and $B_{K,N}:=\frac KN Z_{A_K}Z_{A_K}^\top$. Then $|A_K|/N\to p_K$ and, by Tool S.3 (and its Gram form, which gives full column rank of $Z_{A_K}$ with probability tending to one), the smallest positive eigenvalue of $B_{K,N}$ converges in probability to
\[
K\big(\sqrt c-\sqrt{p_K}\big)^2\ \ge\ K\Big(\sqrt c-\frac{\sqrt c}{2}\Big)^2=\frac{Kc}{4}\ >\ L,
\]
while with probability tending to one its nonzero eigenvalues number $|A_K|$, a fraction $|A_K|/d\to p_K/c>0$ of the dimension. Since $S_{r,N}\succeq B_{K,N}$ (PSD order), Weyl monotonicity places at least $|A_K|$ eigenvalues of $S_{r,N}$ above $a_K:=K(\sqrt c-\sqrt{p_K})^2$ with high probability. Now fix $L<\ell<a_K$ and take bounded continuous $f$ with $f(x)=0$ for $x\le L$, $f(x)=1$ for $x\ge\ell$, $0\le f\le1$. With probability tending to one,
\[
\int f\,d\mu_N\ \ge\ \frac{|A_K|}{d},
\]
so by weak convergence $\int f\,d\mu_Q\ge p_K/c>0$; and since $f\le\mathbf 1_{(L,\infty)}$, $\mu_Q((L,\infty))>0$. As $L$ was arbitrary, $\sup\operatorname{supp}\mu_Q=\infty$; the same argument gives $\lambda_{\max}(S_{r,N})\xrightarrow{\mathbb P}\infty$ in probability. Whether the support can split into several intervals for specific $\rho$ is left open. $\square$

\section{Proof of Theorem~\ref{thm:res-count} (Residual count law; including concentration and all branches)}\label{sec:S14}

\textbf{Setup.} $n_i=K_i+T_i$ with $K_i=\lfloor\lambda_i\rfloor$ deterministic and, when $R_N=\sum_iU_i>0$, $\boldsymbol T\sim\mathrm{Multinomial}(R_N;U/R_N)$; when $R_N=0$, $\boldsymbol T=0$. By (A4R), $\pi_N\Rightarrow\pi$ on $\mathbb Z_{\ge0}\times[0,1]$ and $R_N/N\to\E_\pi U$.

\textbf{Weak convergence, case $\E_\pi U>0$.} Apply the coupling lemma of Section~\ref{sec:S5} with $(\mu_i)=(U_i)$, $R=R_N$: the residual counts couple to independent $\tilde n_i'\sim\operatorname{Poi}(U_i)$ with rank defect $O(\sqrt{R_N})=o(N)$; the diffuse condition on residual weights is automatic ($\max_iU_i/R_N\to0$). For bounded $f:\mathbb Z_{\ge0}\to\mathbb R$, set $g_f(k,u):=\mathbb E[f(k+\operatorname{Poi}(u))]$, which is bounded and continuous on $\mathbb Z_{\ge0}\times[0,1]$; hence by (A4R), $\frac1N\sum_ig_f(K_i,U_i)\to\int g_f\,d\pi$—this is the precise point at which the joint convergence $\pi_N\Rightarrow\pi$ (and not merely $\rho_N\Rightarrow\rho$) is needed. Concentration follows from the bounded-differences argument on the residual ancestor indices. Hence $\nu_N\Rightarrow\mathcal L(K+P_U)$.

\textbf{Weak convergence, case $\E_\pi U=0$.} Then $R_N/N\to0$, and
\[
\#\{i:n_i\neq K_i\}\ \le\ R_N\ =\ o(N),
\]
so for bounded $f$, $\big|\frac1N\sum_if(n_i)-\frac1N\sum_if(K_i)\big|\le2\|f\|_\infty R_N/N\to0$; by (A4R), $\pi_N\Rightarrow\pi$, and since $\E_\pi U=0$ forces $U=0$ $\pi$-a.s., $\mathcal L(K)=\mathcal L(K+U)=\rho$.

\textbf{Second-moment convergence (all regimes).} The following expectation and variance bounds hold for every $R_N>0$, regardless of whether $R_N/N$ tends to zero. Decompose
\[
\frac1N\sum_in_i^2=\frac1N\sum_iK_i^2+\frac2N\sum_iK_iT_i+\frac1N\sum_iT_i^2.
\]
The exact expectation (for $R_N>0$) is
\[
\frac1N\sum_i\mathbb E n_i^2=\frac1N\sum_i\lambda_i^2+\frac{R_N}{N}-\frac1{NR_N}\sum_iU_i^2,\qquad \Big|\frac1{NR_N}\sum_iU_i^2\Big|\le\frac1N,
\]
since $\mathbb ET_i^2=U_i(1-U_i/R_N)+U_i^2$ and $K_i^2+2K_iU_i+U_i^2=\lambda_i^2$. For the variance: multinomial variances satisfy $\operatorname{Var}(T_i)\le U_i$ and cross-covariances are nonpositive, so with $q_i=U_i/R_N$,
\[
\operatorname{Var}\Big(\frac2N\sum_iK_iT_i\Big)\ \le\ \frac4{N^2}\sum_iU_iK_i^2\ \le\ \frac4{N^2}\sum_i\lambda_i^2\ =\ O(N^{-1}),
\]
and the collision computation of S.\ref{sec:S4} with $R_N$ ancestor draws and normalization $N$ (using $\sum_iq_i^2\le1/R_N$, $\sum_iq_i^3\le1/R_N^2$) gives
\[
\operatorname{Var}\Big(\frac1N\sum_iT_i^2\Big)\ \le\ C\,\frac{R_N}{N^2}\ \le\ \frac CN.
\]
The two random terms are not independent, but independence is not needed: $\operatorname{Var}(A+B)\le2\operatorname{Var}(A)+2\operatorname{Var}(B)$. Hence $\frac1N\sum_in_i^2\to r_2+\E_\pi U$ in probability in all regimes ($R_N=0$ is deterministic and immediate). $\square$

\section{Proof of Theorem~\ref{thm:res-lsd} (Residual spectral transform)}\label{sec:S15}

$\nu_{\rm res}$ has mean $\E[K+U]=1$ and finite second moment $r_2+\E_\pi U$. Silverstein's theorem transfers to a random diagonal $D_N$ that is independent of $Z_N$ with converging empirical law—it does not require the diagonal entries to be independent among themselves. Hence one may work with the raw residual count diagonal directly: (i) truncate the \textbf{total} counts, $D_N^{(M)}:=\operatorname{diag}(n_i\wedge M)$ (truncating only the Poisson part would not bound the diagonal, since $K_i$ may be unbounded); (ii) by Theorem~\ref{thm:res-count}, the empirical law of $D_N^{(M)}$ converges to $\nu_{\rm res}^{(M)}$; (iii) conditional on $D_N^{(M)}$, apply the covariance-spectrum theorem (with the subsequence bridge of S.\ref{sec:S6} Step 3); (iv) remove the truncation by the rank bound. This unifies the $\E_\pi U=0$ and $\E_\pi U>0$ cases; Poissonization is retained as a count-law tool only. (Two distinct truncation controls are used and should not be conflated: LSD rank truncation requires only the tail probability $\nu((M,\infty))\to0$; second-tail-moment control is needed for the empirical-moment arguments of S.\ref{sec:S9}.) The Gaussian separable extension (S.\ref{sec:S7}) is unchanged. $\square$

\section{Proof of Corollary~\ref{coro:res-atom} (Residual zero atom and threshold)}\label{sec:S16}

Under $\nu_{\rm res}=\mathcal L(K+P_U)$, a zero count requires $K=0$ and $P_U=0$; on $\{K=0\}$ the latter has probability $e^{-U}$ (with $U=1$ allowed in the limit, contributing $e^{-1}$). Hence $\nu_{\rm res}(\{0\})=\E_\pi[\mathbf 1_{\{K=0\}}e^{-U}]$ and $p_{\rm res}(\pi)=1-\nu_{\rm res}(\{0\})$. Theorem~\ref{thm:rankgap}'s transfer: under the same identity-population, independence, and entry assumptions, the noncritical proof transfers through active-column domination, while the critical proof transfers through the residual LSD and its fixed-point equation. (For the occupancy step, $\mathbf 1_{\{k>0\}}$ is continuous on the discrete count space; if the measures are placed on the real half-axis, replace it by a bounded continuous function agreeing with it at integers.) $\square$

\section{Proof of Proposition~\ref{prop:res-cost} (Residual second-moment cost)}\label{sec:S17}

The count-level identities $\E[K+P_U\mid K,U]=K+U$ and $\Var(K+P_U\mid K,U)=U$ give $b_2(\nu_{\rm res})=\E[(K+U)^2]+\E[U]=r_2+\E_\pi U$. The spectral-moment statement then follows from the residual version of S.\ref{sec:S9}'s argument (the truncation route of S.\ref{sec:S9} applies verbatim once the count second-moment convergence of S.\ref{sec:S14} is available, so the claim is an empirical-moment convergence, not merely a computation of the limit law's second moment). Strictness: $\E_\pi U\le1$ always (equality possible: the endpoint $\pi=\delta_{(0,1)}$), so the inequality is $\le c\,a_1^2$, strict iff $a_1>0$ and $\E_\pi U<1$; for $a_1>0$, equality with the multinomial cost occurs precisely at $\pi=\delta_{(0,1)}$, since $\E_\pi(K+U)=1$ and $\E_\pi U=1$ jointly force $U=1$, $K=0$ $\pi$-a.s. Vanishing: the cost vanishes iff $a_1=0$ or $\E_\pi U=0$; assuming $a_1>0$, iff $\pi(\{U=0\})=1$—a property of the limit law, not of every finite sample (a vanishing fraction of exceptional weights is allowed). Exact finite-$N$ zero cost: whenever counts are deterministic (in particular when every $\lambda_i\in\mathbb Z_{\ge0}$, giving $R_N=0$ and $n_i=\lambda_i$), or under the geometric degeneracy of S.\ref{sec:S3}. $\square$

\section{Proof of Corollary~\ref{coro:rank-order} (Rank-preservation ordering, two levels)}\label{sec:S18}

The multinomial zero mass under the joint description: $\nu_{\rm multi}(\{0\})=\E_\pi[e^{-(K+U)}]=\E_\pi[e^{-L}]$. The residual zero mass (S.\ref{sec:S16}): $\nu_{\rm res}(\{0\})=\E_\pi[\mathbf 1_{\{K=0\}}e^{-U}]$. Difference:
\[
\nu_{\rm multi}(\{0\})-\nu_{\rm res}(\{0\})=\E_\pi\big[\mathbf 1_{\{K\ge1\}}e^{-(K+U)}\big]\ \ge\ 0,
\]
strict iff $\pi(K\ge1)>0$ (the integrand is strictly positive where $\mathbf 1_{\{K\ge1\}}=1$). The spectral ordering $\mu_{\rm res}(\{0\})\le\mu_{\rm multi}(\{0\})$ follows since $q\mapsto\max(0,1-(1-q)/c)$ is nondecreasing. For the spectral strictness, write $q_{\rm multi}:=\big(1-p_{\rm multi}/c\big)_+$ and $h_0:=H(\{0\})$; then $\mu_{\rm res}(\{0\})<\mu_{\rm multi}(\{0\})$ iff $\pi(K\ge1)>0$ and $q_{\rm multi}>h_0$ (in the identity-population case $h_0=0$, reducing to $c>p_{\rm multi}$; for general $H$ the inequality can be masked by a common atom $h_0$). In the window $c\in(p_{\rm multi},p_{\rm res})$, multinomial resampling carries a macroscopic zero atom while residual resampling is of full rank with probability tending to one (identity population). $\square$

\section{Proof of Theorem~\ref{thm:classification} (Classification of residual limits under $\rho=\delta_1$)}\label{sec:S19}

\textbf{Necessity.} If $\rho_N\Rightarrow\delta_1$, then (A4R) forces $\pi$ to be supported on $\{(k,u):k+u=1\}$, since the push-forward under the continuous map $(k,u)\mapsto k+u$ must be $\delta_1$. Since $k\in\mathbb Z_{\ge0}$ and $u\in[0,1]$, the line $k+u=1$ meets the product set only at $(1,0)$ and $(0,1)$. Hence $\pi=(1-\alpha)\delta_{(1,0)}+\alpha\delta_{(0,1)}$ for some $\alpha\in[0,1]$, and Theorem~\ref{thm:res-count} gives $\nu_{\rm res}=(1-\alpha)\delta_1+\alpha\operatorname{Poi}(1)$, with $\E_\pi U=\alpha$, $\nu_{\rm res}(\{0\})=\alpha e^{-1}$, $p_{\rm res}=1-\alpha/e$, and second-moment cost $c\,a_1^2\alpha$ (Proposition~\ref{prop:res-cost}). $\square$

\textbf{Realizability.} (i) $\alpha=0$: exactly uniform weights ($K_i=1$, $U_i=0$, $R_N=0$). (ii) $\alpha=1$ (bounded two-level endpoint construction): let $\lambda_i=1-\frac1N$ for $i<N$ and $\lambda_N=2-\frac1N$. Then $\sum_i\lambda_i=N$ exactly, $\max_i\lambda_i<2$, $\frac1N\sum_i\lambda_i^2=1+\frac1N-\frac1{N^2}\to1$, and $\rho_N\Rightarrow\delta_1$; the first $N-1$ particles have $(\lfloor\lambda_i\rfloor,\operatorname{frac}(\lambda_i))=(0,1-\frac1N)\to(0,1)$, and the last has proportion $1/N\to0$, so $\pi_N\Rightarrow\delta_{(0,1)}$; Theorem~\ref{thm:res-count} then yields $\nu_{\rm res}=\operatorname{Poi}(1)$ directly, with no informal per-particle argument. (iii) $0<\alpha<1$: let $\alpha_N\to\alpha$ with $\alpha_NN$ integral and $\varepsilon_N\downarrow0$; put a fraction $\alpha_N$ of particles at $1-\varepsilon_N$ and the rest at $1+\frac{\alpha_N}{1-\alpha_N}\varepsilon_N$. The mean is exactly one: $\alpha_N(1-\varepsilon_N)+(1-\alpha_N)\big(1+\frac{\alpha_N}{1-\alpha_N}\varepsilon_N\big)=1$; the exact second moment is
\[
\frac1N\sum_i\lambda_i^2=1+\frac{\alpha_N}{1-\alpha_N}\,\varepsilon_N^2,
\]
and the exact residual fraction is $R_N/N=\alpha_N$ (both used in experiment E4); the joint limit is $(1-\alpha)\delta_{(1,0)}+\alpha\delta_{(0,1)}$. All constructions satisfy both (A4) and (A4R). $\square$

\section{Proof of Corollary~\ref{coro:uniform-contrast} (Exact-uniform-weight contrast)}\label{sec:S20}

For $w_i=1/N$ exactly: residual resampling leaves the second-moment matrix unchanged at every finite $N$ ($n_i\equiv1$ deterministically); in the identity-population setting its LSD is therefore $\mathrm{MP}(c)$ (Theorem~\ref{thm:res-lsd} with $\eta=\delta_1$, normalization check of S.\ref{sec:S6}), with zero cost and $p_{\rm res}=1$; in the general Gaussian population model the unchanged matrix has LSD $H\boxtimes\mathrm{MP}(c)$. Multinomial resampling gives $\operatorname{Poi}(1)$, the distorted $\mathcal T_c(\operatorname{Poi}(1))$, cost $c\,a_1^2$, and $p_{\rm multi}=1-e^{-1}$. Both sides have $r_2=1$. The interpolation clause is Theorem~\ref{thm:classification}. $\square$

\section{Lemma (Injectivity of $\mathcal T_c$; self-contained, all nonnegative laws)}\label{sec:S21}

\textbf{Statement.} For every $c>0$, $\mathcal T_c$ is injective on the class of \textbf{all} probability laws on $[0,\infty)$.

\textbf{Proof.} Suppose $\mathcal T_c(\nu)=\mathcal T_c(\nu')$ and let $m$ be the common Stieltjes transform. For every $z$ in the upper half-plane, Definition~\ref{def:Tc}'s fixed-point equation gives
\[
\int\frac{t}{1+ctm(z)}\,\nu(dt)\ =\ \int\frac{t}{1+ctm(z)}\,\nu'(dt).
\]
Take $z=-s<0$ and set $r:=cm(-s)>0$. The map $s\mapsto cm(-s)$ is continuous and strictly decreasing on any interval of the negative real axis ($\frac{d}{ds}m(-s)=-\int\frac{\mu(dx)}{(x+s)^2}<0$), so as $s$ varies over a nondegenerate interval, $r$ varies over a nondegenerate interval $I\subset(0,\infty)$—that is all the proof needs. Define
\[
F_\nu(r):=\int\frac{t}{1+rt}\,\nu(dt)\ =\ \frac1r-\frac1{r^2}\,m_\nu\Big(-\frac1r\Big),
\]
the identity holding since $\frac1r\int\frac{\nu(dx)}{1+rx}=\frac1r-\frac1{r^2}m_\nu(-1/r)$ by direct algebra. Equality of $F_\nu$ and $F_{\nu'}$ on $I$ gives equality of $m_\nu$ and $m_{\nu'}$ on the corresponding negative interval; the identity theorem for analytic functions extends the equality to the common domain of analyticity, and uniqueness of the Stieltjes transform gives $\nu=\nu'$. (This route needs no mean-one assumption, no weighted measure, and no separate recovery of the atom at zero; the earlier cancellation-by-$S$-transform argument is retained only as a remark.) $\square$

\section{Additional References for the Supplement}\label{sec:S22}

[S1] J.~W.~Silverstein and S.~Choi, ``Analysis of the limiting spectral distribution of large dimensional random matrices,'' \emph{J. Multivariate Anal.}, vol.~54, no.~2, pp.~295--309, 1995 (main-text \cite{silverstein1995analysis}; edge parametrization).
[S2] S.~Belinschi, ``The atoms of the free multiplicative convolution of two probability distributions,'' \emph{Integral Equations and Operator Theory}, vol.~46, no.~4, pp.~377--386, 2003 (main-text \cite{belinschi-atoms}; max-atom rule used in S.\ref{sec:S8}).
[S3] Z.~D.~Bai and Y.~Q.~Yin, ``Limit of the smallest eigenvalue of a large dimensional sample covariance matrix,'' \emph{Ann. Probab.}, vol.~21, no.~3, pp.~1275--1294, 1993 (main-text \cite{bai1993limit}; Tool S.3).
[S4] H.~Bercovici and D.~Voiculescu, ``Free convolution of measures with unbounded support,'' \emph{Indiana Univ. Math. J.}, vol.~42, no.~3, pp.~733--773, 1993 (main-text \cite{bercovici1993free}).

\section{Solver Unit Tests}\label{sec:S23}

All spectral predictions in the main text are computed from the fixed-point
equation of Definition~\ref{def:Tc} by two independent implementations (covariance side
and companion side, the latter via the auxiliary variable $u=c\,m$ and
$e=u+(c-1)/z$). The battery below is run before any experiment; each row
reports the realized tolerance.

\begin{center}
\renewcommand{\arraystretch}{1.25}
\begin{tabular}{llll}
\toprule
Check & Quantity & Realized & Criterion \\
\midrule
Branch & $\operatorname{Im}m(z)$ at $z=\mathrm{i},\,3+2\mathrm{i},\,0.05+0.5\mathrm{i}$
  & $0.610,\ 0.263,\ 0.919$ & $>0$ \\
Equation residual & $\max$ over grid & $2.2\times10^{-13}$ & $\ll 10^{-8}$ \\
Normalization & $z\,m(z)$ at $z=10^{6}\mathrm{i}$ & $-1.000000+10^{-6}\mathrm{i}$ & $\to-1$ \\
$\mathrm{MP}(c)$ recovery & max interior relative density error & $0.760\%$ & $<1\%$ \\
Zero-mass consistency & max $|$solver $-$ analytic$|$, three configurations
  & $4.3\times10^{-8}$ & $\ll10^{-6}$ \\
Truncation $K{=}20$ vs $K{=}60$ & $\max|\Delta f|$ on grid & $2.1\times10^{-15}$ & $\ll10^{-8}$ \\
 & discarded second tail moment & $5.1\times10^{-15}$ & $\ll10^{-8}$ \\
Dual-solver cross-check & transform-relation residual & $6.4\times10^{-13}$ & $\ll10^{-8}$ \\
 & companion identity $\int(1+tu)^{-1}=-z\,e$ residual & $1.9\times10^{-13}$ & $\ll10^{-8}$ \\
 & density deviation, $x>0.15$ & $0.000\%$ & $<2\%$ \\
\bottomrule
\end{tabular}
\end{center}

The bandwidth sensitivity $\max|\Delta f|=0.1438$ at $x=0.020$ between
$h=10^{-3}$ and $h=5\times10^{-4}$ is edge sharpening at the left boundary
layer and is reported rather than tested; all conclusions in the main text are
drawn away from this layer, and the $h/2$ sensitivity check reported there
leaves every conclusion unchanged. The classical companion fixed-point forms
found in the literature govern the companion transform of the $W$-side model
and do not apply to the object studied here; the cross-check above is the
reason the two implementations cannot silently agree on a wrong equation.

\section{Experiment E2: Full Twelve-Point Sweep}\label{sec:S24}

Exactly uniform weights, so that $p_{\mathrm{multi}}=1-e^{-1}$ and the
threshold is $c=1-e^{-1}\approx0.6321$. Part~A is count-only (exact rank
identity, $N=10^{4}$, $300$ repetitions); Part~B reports the median and
quartiles of $\lambda_{\min}^{+}$ over $15$ repetitions at $N=1000$, against
the Marchenko--Pastur lower-edge benchmark $(\sqrt c-\sqrt p)^2$ of
Theorem~\ref{thm:rankgap} and the true lower edge computed from the
critical equation of Proposition~\ref{prop:edges}.

\begin{center}
\renewcommand{\arraystretch}{1.15}
\begin{tabular}{cccc}
\toprule
$c$ & $q_{0,N}$ empirical & $q_0$ theory & $|$gap$|$ \\
\midrule
$0.35$   & $0$                  & $0$      & $0$ \\
$0.45$   & $0$                  & $0$      & $0$ \\
$0.55$   & $0$                  & $0$      & $0$ \\
$0.60$   & $0$                  & $0$      & $0$ \\
$0.6321$ & $0.00203\pm0.00017$  & $0$      & critical band \\
$0.66$   & $0.04223\pm0.00027$  & $0.04224$ & $1\times10^{-5}$ \\
$0.72$   & $0.12203\pm0.00024$  & $0.12205$ & $3\times10^{-5}$ \\
$0.80$   & $0.20997\pm0.00023$  & $0.20985$ & $1\times10^{-4}$ \\
$0.90$   & $0.29750\pm0.00020$  & $0.29764$ & $1\times10^{-4}$ \\
$1.00$   & $0.36803\pm0.00019$  & $0.36788$ & $2\times10^{-4}$ \\
$1.08$   & $0.41480\pm0.00018$  & $0.41470$ & $1\times10^{-4}$ \\
$1.15$   & $0.45029\pm0.00016$  & $0.45033$ & $4\times10^{-5}$ \\
\bottomrule
\end{tabular}
\end{center}

The nonzero reading $0.00203\pm0.00017$ at the critical point $c=0.6321$ is
the finite-size critical-band smoothing annotated in the main text; at every
off-critical grid point the empirical zero mass tracks
$\max\big(0,1-(1-e^{-1})/c\big)$ to within $\pm0.001$.

\begin{center}
\renewcommand{\arraystretch}{1.15}
\begin{tabular}{ccccc}
\toprule
$c$ & $\lambda_{\min}^{+}$ median & quartiles & MP benchmark & true edge \\
\midrule
$0.35$   & $0.05906$ & $[0.05517,\,0.06203]$ & $0.04139$ & $0.05607$ \\
$0.45$   & $0.02166$ & $[0.02018,\,0.02349]$ & $0.01544$ & $0.02058$ \\
$0.55$   & $0.00399$ & $[0.00341,\,0.00471]$ & $0.00286$ & $0.00376$ \\
$0.60$   & $0.00089$ & $[0.00071,\,0.00106]$ & $0.00042$ & $0.00055$ \\
$0.6321$ & $8.5\times10^{-5}$ & $[2.4\times10^{-5},\,1.4\times10^{-4}]$ & $0$ & $0$ \\
$0.66$   & $0.00037$ & $[0.00019,\,0.00054]$ & $0.00030$ & $0.00039$ \\
$0.72$   & $0.00359$ & $[0.00316,\,0.00432]$ & $0.00286$ & $0.00369$ \\
$0.80$   & $0.01433$ & $[0.01313,\,0.01499]$ & $0.00987$ & $0.01267$ \\
$0.90$   & $0.03050$ & $[0.02884,\,0.03126]$ & $0.02360$ & $0.03004$ \\
$1.00$   & $0.05428$ & $[0.04954,\,0.05957]$ & $0.04200$ & $0.05308$ \\
$1.08$   & $0.07949$ & $[0.07602,\,0.08137]$ & $0.05962$ & $0.07495$ \\
$1.15$   & $0.09721$ & $[0.09564,\,0.10371]$ & $0.07691$ & $0.09627$ \\
\bottomrule
\end{tabular}
\end{center}

The median tracks the true edge at every grid point and collapses at
criticality ($8.5\times10^{-5}$ at $c=p$), as Theorem~\ref{thm:rankgap}(iii) requires. The
true edge exceeds the benchmark at every off-critical point, by a factor
between $1.25$ and $1.27$ over the upper half of the grid---the
``approximately $1.3$'' conservativeness quoted in the main text.

\section{Experiment E1: Count Probability Mass Functions}\label{sec:S25}

Count probability mass functions for the design of Section~\ref{sec:ix-e1} of the
main text ($N=1000$, $c=0.8$), reported here at $200$ paired repetitions as an
independent higher-replication check of the same pre-registered design; the
uniform--residual row is degenerate ($n_i\equiv1$) and is omitted.

\begin{center}
\renewcommand{\arraystretch}{1.15}
\begin{tabular}{ccccccc}
\toprule
 & \multicolumn{2}{c}{uniform, multinomial}
 & \multicolumn{2}{c}{two-point, multinomial}
 & \multicolumn{2}{c}{two-point, residual} \\
\cmidrule(lr){2-3}\cmidrule(lr){4-5}\cmidrule(lr){6-7}
$k$ & theory & empirical & theory & empirical & theory & empirical \\
\midrule
$0$ & $0.36788$ & $0.36764\pm0.00068$ & $0.41483$ & $0.41441\pm0.00071$ & $0.30327$ & $0.30487\pm0.00065$ \\
$1$ & $0.36788$ & $0.36704\pm0.00110$ & $0.31898$ & $0.31982\pm0.00104$ & $0.45490$ & $0.45228\pm0.00111$ \\
$2$ & $0.18394$ & $0.18602\pm0.00080$ & $0.16342$ & $0.16302\pm0.00075$ & $0.18954$ & $0.19010\pm0.00076$ \\
$3$ & $0.06131$ & $0.06051\pm0.00048$ & $0.06907$ & $0.06903\pm0.00050$ & $0.04423$ & $0.04446\pm0.00041$ \\
$4$ & $0.01533$ & $0.01516\pm0.00024$ & $0.02432$ & $0.02436\pm0.00032$ & $0.00711$ & $0.00736\pm0.00019$ \\
$5$ & $0.00307$ & $0.00312\pm0.00012$ & $0.00714$ & $0.00714\pm0.00018$ & $0.00087$ & $0.00085\pm0.00006$ \\
\bottomrule
\end{tabular}
\end{center}

The theoretical columns are $\operatorname{Poi}(1)$,
$\operatorname{MixPoi}(\tfrac12\delta_{1/2}+\tfrac12\delta_{3/2})$, and
$\mathcal L(K+P_U)$ with $(K,U)$ uniform on $\{(0,\tfrac12),(1,\tfrac12)\}$
respectively, exactly as Theorems~\ref{thm:multi-count}
and~\ref{thm:res-count} predict; the zero row reproduces the
zero masses of Table~\ref{tab:e1} of the main text within the respective error bars.

\section{Experiment E4: Full Sweep}\label{sec:S26}

Full sweep $N=2^{8},\dots,2^{14}$ for the three constructions of Theorem~\ref{thm:classification},
extending Table~\ref{tab:e4} of the main text (which reports the endpoints). Columns:
residual randomization fraction $R_N/N$ (theory $\alpha$), zero-count fraction
(theory $\alpha e^{-1}$), count second moment $b_{2,N}$ (theory $1+\alpha$),
and second-order weight deviation. For the uniform construction the counts are
deterministic and all entries are exact.

\begin{center}
\renewcommand{\arraystretch}{1.15}
\begin{tabular}{lccccc}
\toprule
Construction & $N$ & $R_N/N$ & zero fraction & $b_{2,N}$ & $\frac1N\sum(\lambda_i-1)^2$ \\
\midrule
interior & $256$   & $0.3984$ & $0.15408\pm0.00089$ & $1.3963\pm0.0037$ & $2.587\times10^{-3}$ \\
($\alpha=0.4$) & $512$   & $0.4004$ & $0.1536\pm0.0006$ & $1.3997\pm0.0027$ & $1.304\times10^{-3}$ \\
 & $1024$  & $0.4004$ & $0.1525\pm0.0005$ & $1.4046\pm0.0020$ & $6.521\times10^{-4}$ \\
 & $2048$  & $0.3999$ & $0.1501\pm0.0003$ & $1.3993\pm0.0014$ & $3.254\times10^{-4}$ \\
 & $4096$  & $0.3999$ & $0.1494\pm0.0002$ & $1.3997\pm0.0010$ & $1.627\times10^{-4}$ \\
 & $8192$  & $0.4000$ & $0.1488\pm0.0002$ & $1.4000\pm0.0007$ & $8.139\times10^{-5}$ \\
 & $16384$ & $0.4000$ & $0.1483\pm0.0001$ & $1.3998\pm0.0005$ & $4.069\times10^{-5}$ \\
\addlinespace[2pt]
endpoint & $256$   & $0.9961$ & $0.36859\pm0.00135$ & $2.0034\pm0.0062$ & $3.891\times10^{-3}$ \\
($\alpha=1$) & $512$   & $0.9980$ & $0.3681\pm0.0010$ & $1.9991\pm0.0043$ & $1.949\times10^{-3}$ \\
 & $1024$  & $0.9990$ & $0.3669\pm0.0007$ & $1.9967\pm0.0030$ & $9.756\times10^{-4}$ \\
 & $2048$  & $0.9995$ & $0.3680\pm0.0005$ & $2.0003\pm0.0021$ & $4.880\times10^{-4}$ \\
 & $4096$  & $0.9998$ & $0.3678\pm0.0003$ & $1.9979\pm0.0015$ & $2.441\times10^{-4}$ \\
 & $8192$  & $0.9999$ & $0.3675\pm0.0002$ & $1.9998\pm0.0011$ & $1.221\times10^{-4}$ \\
 & $16384$ & $0.9999$ & $0.3680\pm0.0002$ & $2.0007\pm0.0008$ & $6.103\times10^{-5}$ \\
\addlinespace[2pt]
uniform & all & $0$ & $0$ & $1$ & $0$ \\
\midrule
\multicolumn{2}{l}{theory} & $\alpha$ & $\alpha e^{-1}$ & $1+\alpha$ & $\to0$ \\
\bottomrule
\end{tabular}
\end{center}

The zero-count fraction converges to $\alpha e^{-1}$ from above at rate
consistent with $N^{-1/2}$, and the second-order deviation decreases by a
factor of $4$ per doubling of $N$ in all three constructions, while $R_N/N$
remains at its theoretical limit throughout.

\section{Robustness Checks in Full}\label{sec:S27}

\textbf{Scale checks (A1).} E1 repeated at $N=500$ and $N=2000$ (same $c=0.8$,
$20$ repetitions):

\begin{center}
\renewcommand{\arraystretch}{1.15}
\begin{tabular}{llccc}
\toprule
$N$ & profile & $m_2$ (th) & $\nu(\{0\})$ (th) & $\mu_Q(\{0\})$ (th) \\
\midrule
$500$  & uniform   & $2.5860\pm0.0210$ ($2.6$) & $0.3646\pm0.0051$ ($0.36788$) & $0.20575\pm0.00637$ ($0.20985$) \\
$500$  & two-point & $2.8327\pm0.0142$ ($2.8$) & $0.4154\pm0.0042$ ($0.41483$) & $0.26925\pm0.00525$ ($0.26854$) \\
$2000$ & uniform   & $2.6184\pm0.0081$ ($2.6$) & $0.3701\pm0.0020$ ($0.36788$) & $0.21256\pm0.00244$ ($0.20985$) \\
$2000$ & two-point & $2.7992\pm0.0102$ ($2.8$) & $0.4145\pm0.0020$ ($0.41483$) & $0.26806\pm0.00247$ ($0.26854$) \\
\bottomrule
\end{tabular}
\end{center}

Every quantity converges toward its prediction with error bars shrinking at
the Monte Carlo rate; deviations at $N=2000$ are within $1.4\sigma$
everywhere.

\textbf{Non-Gaussian entries (A2).} Two-point profile, multinomial
resampling, $c=0.8$, variance-one entries with kurtosis $\kappa_4$; finite-$N$
predictions from Proposition~\ref{prop:popcost} include the kurtosis term:

\begin{center}
\renewcommand{\arraystretch}{1.15}
\begin{tabular}{lcccc}
\toprule
entries & $\kappa_4$ & $m_2$ empirical & $m_2$ finite-$N$ theory & $\mu_Q(\{0\})$ empirical (th $0.26854$) \\
\midrule
Gaussian    & $0$    & $2.800\pm0.012$ & $2.80000$ & $0.2705\pm0.0052$ \\
Rademacher  & $-2$   & $2.821\pm0.017$ & $2.79800$ & $0.2755\pm0.0047$ \\
Student-$t_8$ & $1.5$ & $2.821\pm0.021$ & $2.80150$ & $0.2690\pm0.0051$ \\
\bottomrule
\end{tabular}
\end{center}

Deviations from the finite-$N$ predictions are $1.3\sigma$ and $0.9\sigma$.
This is an agreement report for a single configuration; no universality claim
is made.

\textbf{Separable population spectrum (A3).}
$H=\tfrac12\delta_{1/2}+\tfrac12\delta_{3/2}$, uniform weights, multinomial
resampling, $c=0.8$: post-resampling $m_1=1.0011$ (theory $a_1=1$); $m_2$ pre
$2.054$ (theory $2.05$) and post $2.856$ (theory $2.85$); zero atom $0.2079$
(theory $0.20985$ by the max-atom rule of Proposition~\ref{prop:atom}). As an end-to-end
control, the $H\boxtimes\mathcal T_c(\operatorname{Poi}(1))$ pipeline
(S-transform computation) reproduces a directly simulated separable Wishart
spectrum with maximum Stieltjes deviation $5.1\times10^{-4}$
(Fig.~\ref{fig:a3}).

\begin{figure}[!t]
  \centering
  \includegraphics[width=\columnwidth]{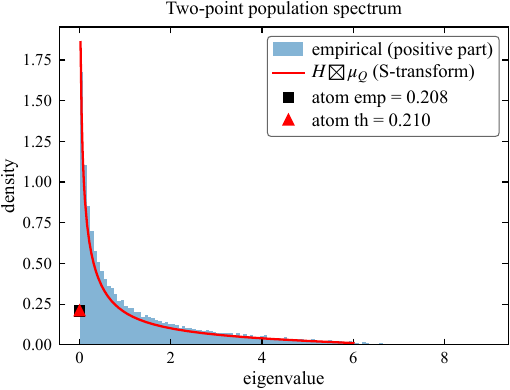}
  \caption{End-to-end check of Theorem~\ref{thm:lsd-gen}: directly simulated spectrum under
    the two-point population $H=\frac12\delta_{1/2}+\frac12\delta_{3/2}$
    against the $H\boxtimes\mathcal T_c(\operatorname{Poi}(1))$ prediction,
    with the zero atom drawn separately. The two computation routes agree to
    a maximum Stieltjes deviation of $5.1\times10^{-4}$.}
  \label{fig:a3}
\end{figure}

\bibliographystyle{IEEEtran}
\bibliography{dm_refs} 

@incollection{doucet2011tutorial,
  author    = {Arnaud Doucet and Adam M. Johansen},
  title     = {A tutorial on particle filtering and smoothing: fifteen years later},
  booktitle = {The Oxford Handbook of Nonlinear Filtering},
  editor    = {Dan Crisan and Boris Rozovskii},
  publisher = {Oxford University Press},
  year      = {2011},
  pages     = {656--704}
}

@book{delmoral2004feynman,
  author    = {Pierre Del Moral},
  title     = {Feynman--Kac Formulae: Genealogical and Interacting Particle Systems with Applications},
  series    = {Probability and Its Applications},
  publisher = {Springer},
  year      = {2004}
}

@article{kong1994sequential,
  author  = {Augustine Kong and Jun S. Liu and Wing Hung Wong},
  title   = {Sequential imputations and {Bayesian} missing data problems},
  journal = {Journal of the American Statistical Association},
  volume  = {89},
  number  = {425},
  pages   = {590--599},
  year    = {1994}
}

@article{gerber2019negative,
  author  = {Mathieu Gerber and Nicolas Chopin and Nick Whiteley},
  title   = {Negative association, ordering and convergence of resampling methods},
  journal = {The Annals of Statistics},
  volume  = {47},
  number  = {4},
  pages   = {2236--2260},
  year    = {2019}
}

@article{elvira2022rethinking,
  author  = {V\'ictor Elvira and Luca Martino and Christian P. Robert},
  title   = {Rethinking the effective sample size},
  journal = {International Statistical Review},
  volume  = {90},
  number  = {1},
  pages   = {29--53},
  year    = {2022}
}

@article{silverstein1995strong,
  author  = {Jack W. Silverstein},
  title   = {Strong convergence of the empirical distribution of eigenvalues of large dimensional random matrices},
  journal = {Journal of Multivariate Analysis},
  volume  = {55},
  number  = {2},
  pages   = {331--339},
  year    = {1995}
}

@book{bai2010spectral,
  author    = {Zhidong Bai and Jack W. Silverstein},
  title     = {Spectral Analysis of Large Dimensional Random Matrices},
  edition   = {2},
  series    = {Springer Series in Statistics},
  publisher = {Springer},
  year      = {2010}
}

@article{hachem2007deterministic,
  author  = {Walid Hachem and Philippe Loubaton and Jamal Najim},
  title   = {Deterministic equivalents for certain functionals of large random matrices},
  journal = {The Annals of Applied Probability},
  volume  = {17},
  number  = {3},
  pages   = {875--930},
  year    = {2007}
}

@article{voiculescu1987multiplication,
  author  = {Dan Voiculescu},
  title   = {Multiplication of certain noncommuting random variables},
  journal = {Journal of Operator Theory},
  volume  = {18},
  number  = {2},
  pages   = {223--235},
  year    = {1987}
}

@article{collins2006integration,
  author  = {Beno\^{\i}t Collins and Piotr \'Sniady},
  title   = {Integration with respect to the {Haar} measure on unitary, orthogonal and symplectic group},
  journal = {Communications in Mathematical Physics},
  volume  = {264},
  number  = {3},
  pages   = {773--795},
  year    = {2006}
}

@article{benaych2009rectangular,
  author  = {Florent Benaych-Georges},
  title   = {Rectangular random matrices, related convolution},
  journal = {Probability Theory and Related Fields},
  volume  = {144},
  number  = {3--4},
  pages   = {471--515},
  year    = {2009}
}

@article{davis1957schwarz,
  author  = {Chandler Davis},
  title   = {A {Schwarz} inequality for convex operator functions},
  journal = {Proceedings of the American Mathematical Society},
  volume  = {8},
  number  = {1},
  pages   = {42--44},
  year    = {1957}
}

@article{fan1949theorem,
  author  = {Ky Fan},
  title   = {On a theorem of {Weyl} concerning eigenvalues of linear transformations. {I}},
  journal = {Proceedings of the National Academy of Sciences},
  volume  = {35},
  number  = {11},
  pages   = {652--655},
  year    = {1949}
}

@article{karamata1932inegalite,
  author  = {Jovan Karamata},
  title   = {Sur une in\'egalit\'e relative aux fonctions convexes},
  journal = {Publications de l'Institut Math\'ematique (Belgrade)},
  volume  = {1},
  pages   = {145--148},
  year    = {1932}
}

@book{marshall1979inequalities,
  author    = {Albert W. Marshall and Ingram Olkin and Barry C. Arnold},
  title     = {Inequalities: Theory of Majorization and Its Applications},
  publisher = {Springer},
  year      = {1979},
  note      = {Second edition, 2011}
}

@book{bhatia1997matrix,
  author    = {Rajendra Bhatia},
  title     = {Matrix Analysis},
  series    = {Graduate Texts in Mathematics 169},
  publisher = {Springer},
  year      = {1997}
}

@book{rockafellar1970convex,
  author    = {R. Tyrrell Rockafellar},
  title     = {Convex Analysis},
  publisher = {Princeton University Press},
  year      = {1970}
}

@article{elkaroui2016bootstrap,
  author  = {Noureddine El Karoui and Elizabeth Purdom},
  title   = {The bootstrap, covariance matrices and {PCA} in moderate and high-dimensions},
  journal = {arXiv preprint arXiv:1608.00948},
  year    = {2016}
}

@article{li2012deterministic,
  author  = {Tiancheng Li and Tariq Parvez Sattar and Shudong Sun},
  title   = {Deterministic resampling: unbiased sampling to avoid sample impoverishment in particle filters},
  journal = {Signal Processing},
  volume  = {90},
  number  = {1},
  pages   = {279--293},
  year    = {2012}
}

@article{hansen1982jensen,
  author  = {Frank Hansen and Gert K. Pedersen},
  title   = {Jensen's inequality for operators and {L\"owner}'s theorem},
  journal = {Mathematische Annalen},
  volume  = {258},
  number  = {3},
  pages   = {229--241},
  year    = {1982}
}

@article{silverstein1995analysis,
  author  = {Jack W. Silverstein and Sang Il Choi},
  title   = {Analysis of the limiting spectral distribution of large dimensional random matrices},
  journal = {Journal of Multivariate Analysis},
  volume  = {54},
  number  = {2},
  pages   = {295--309},
  year    = {1995}
}

@article{bai1993limit,
  author  = {Z. D. Bai and Y. Q. Yin},
  title   = {Limit of the smallest eigenvalue of a large dimensional sample covariance matrix},
  journal = {The Annals of Probability},
  volume  = {21},
  number  = {3},
  pages   = {1275--1294},
  year    = {1993}
}

@phdthesis{zhang2006spectral,
  author = {Lixin Zhang},
  title  = {Spectral analysis of large dimensional random matrices},
  school = {National University of Singapore},
  year   = {2006}
}

@article{paul2009no,
  author  = {Debashis Paul and Jack W. Silverstein},
  title   = {No eigenvalues outside the support of the limiting empirical spectral distribution of a separable covariance matrix},
  journal = {Journal of Multivariate Analysis},
  volume  = {100},
  number  = {1},
  pages   = {37--57},
  year    = {2009}
}

@article{bercovici1993free,
  author  = {Hari Bercovici and Dan Voiculescu},
  title   = {Free convolution of measures with unbounded support},
  journal = {Indiana University Mathematics Journal},
  volume  = {42},
  number  = {3},
  pages   = {733--773},
  year    = {1993}
}

@article{belinschi-atoms,
  author  = {Serban Belinschi},
  title   = {The atoms of the free multiplicative convolution of two probability distributions},
  journal = {Integral Equations and Operator Theory},
  volume  = {46},
  number  = {4},
  pages   = {377--386},
  year    = {2003}
}

@article{gordon1993novel,
  author  = {Neil J. Gordon and David J. Salmond and Adrian F. M. Smith},
  title   = {Novel approach to nonlinear/non-{G}aussian {B}ayesian state
             estimation},
  journal = {IEE Proceedings F (Radar and Signal Processing)},
  volume  = {140},
  number  = {2},
  pages   = {107--113},
  year    = {1993}
}

@article{cappe2007overview,
  author  = {Olivier Capp\'e and Simon J. Godsill and Eric Moulines},
  title   = {An overview of existing methods and recent advances in
             sequential {M}onte {C}arlo},
  journal = {Proceedings of the IEEE},
  volume  = {95},
  number  = {5},
  pages   = {899--924},
  year    = {2007}
}

@inproceedings{douc2005comparison,
  author    = {Randal Douc and Olivier Capp\'e and Eric Moulines},
  title     = {Comparison of resampling schemes for particle filtering},
  booktitle = {Proc. 4th Int. Symp. Image and Signal Processing and Analysis
               (ISPA)},
  address   = {Zagreb, Croatia},
  pages     = {64--69},
  year      = {2005}
}

@article{martino2017ess,
  author  = {Luca Martino and V\'ictor Elvira and Francisco Louzada},
  title   = {Effective sample size for importance sampling based on
             discrepancy measures},
  journal = {Signal Processing},
  volume  = {131},
  pages   = {386--401},
  year    = {2017}
}

@article{marchenko1967distribution,
  author  = {V. A. Marchenko and L. A. Pastur},
  title   = {Distribution of eigenvalues for some sets of random matrices},
  journal = {Mathematics of the USSR-Sbornik},
  volume  = {1},
  number  = {4},
  pages   = {457--483},
  year    = {1967}
}

@article{baik2005phase,
  author  = {Jinho Baik and G\'erard Ben Arous and Sandrine P\'ech\'e},
  title   = {Phase transition of the largest eigenvalue for nonnull complex
             sample covariance matrices},
  journal = {The Annals of Probability},
  volume  = {33},
  number  = {5},
  pages   = {1643--1697},
  year    = {2005}
}

\end{document}